\documentclass[12pt]{amsart}
\usepackage[letterpaper,margin=1.1in]{geometry}
\usepackage{amsmath}
\usepackage{amssymb}
\usepackage{curves}
\usepackage[french,english]{babel}
\usepackage{epic}
\usepackage{rotating}
\usepackage{epsf}
\usepackage[scanall]{psfrag}
\usepackage{esint} 
\usepackage{enumerate}
\usepackage[nobysame,alphabetic]{amsrefs}

\usepackage{hyperref} 
\hypersetup{backref,  pdfpagemode=FullScreen,  colorlinks=true} 

\numberwithin{equation}{section}

\newtheorem{theorem}{Theorem}[section]

\newtheorem{corollary}{Corollary}[section]
\newtheorem{lemma}{Lemma}[section]

\newtheorem{defn}{Definition}[section]

\theoremstyle{remark}

\newtheorem{remark}{Remark}[section]

\newcommand{\R}{\mathbb{R}}

\renewcommand{\O}{\Omega}
\newcommand{\p}{\partial}
\newcommand{\loc}{\mathrm{loc}}

\newcommand{\dist}{\mathrm{dist}\,}

\newcommand{\sm}{\setminus}

\newcommand{\wt}{\widetilde}

\usepackage{xcolor}

\begin{document}

\title[Variable coefficient regularity]{Regularity for almost-minimizers of variable coefficient Bernoulli-type functionals}
\author{Guy David}
\author{Max Engelstein}
\author{Mariana Smit Vega Garcia}
\author{Tatiana Toro}
  \thanks{GD was partially supported by the ANR, programme blanc GEOMETRYA, ANR-12-BS01-0014, 
 the European H2020 Grant GHAIA 777822, and the Simons Collaborations in MPS Grant 601941, GD.
 ME was partially supported by NSF DMS 2000288, NSF DMS 1703306 and by David Jerison's grant NSF DMS 1500771. TT was partially supported by  NSF grant DMS-1664867, % and 
 by the Craig McKibben \& Sarah Merner Professorship in Mathematics, 
 and by the Simons Foundation Fellowship 614610.} %NN added grant
\address{
Laboratoire de Math\'ematiques d?Orsay, Univ. Paris-Sud, CNRS, Universit\'e Paris-Saclay, 
91405 Orsay, France.}
\email{Guy.David@math.u-psud.fr}
\address{School of Mathematics,\\ University of Minnesota,\\ Minneapolis,
MN, 55455, USA.}
\email{mengelst@umn.edu}
\address{Western Washington University \\ Department of Mathematics \\
BH 230\\
Bellingham, WA 98225, USA}
\email{Mariana.SmitVegaGarcia@wwu.edu}

\address{Department of Mathematics\\ University of Washington\\ Box 354350\\ Seattle, WA 98195-4350}
\email{toro@uw.edu}
\date{\today}

\maketitle

\vspace{-.7cm}
\begin{abstract}
In \cite{davidtoroalmostminimizers} and \cite{davidengelsteintoro}, 
the authors studied almost minimizers 
for functionals of the type first studied by Alt and Caffarelli in \cite{AC} and Alt, Caffarelli and Friedman in \cite{ACF}. 
 In this paper we study the regularity of almost minimizers to energy functionals with variable coefficients (as opposed to \cite{davidtoroalmostminimizers}, \cite{davidengelsteintoro}, \cite{AC} and \cite{ACF} which deal only with the ``Laplacian" setting). We prove Lipschitz regularity up to, and across, the free boundary, fully %NN fully?
 generalizing the results of \cite{davidtoroalmostminimizers} to the variable coefficient setting.
\end{abstract}

\selectlanguage{french}
\begin{abstract} 
Dans \cite{davidtoroalmostminimizers} et \cite{davidengelsteintoro}, les auteurs ont \'etudi\'e les fonctions 
presque minimales pour des fonctionnelles comme celles d'Alt et Caffarelli \cite{AC}, 
et d'Alt, Caffarelli et Friedman \cite{ACF}. Dans ce papier on \'etudie la r\'egularit\'e des fonctions presque minimales
pour des fonctionnelles d'\'energie \`a coefficients variables (contrairement \`a \cite{davidtoroalmostminimizers}, \cite{davidengelsteintoro}, \cite{AC} et \cite{ACF} qui se placent dans le cadre du Laplacian). On prouve que ces fonctions sont Lipschitziennes juqu'\`a la fronti\`ere, et \`a travers, g\'en\'eralisant ainsi %NN fully?
les r\'esultats de \cite{davidtoroalmostminimizers} au cas de coefficients variables.
\end{abstract}
\selectlanguage{english}

\section{Introduction}

In \cite{davidtoroalmostminimizers} and \cite{davidengelsteintoro}, the authors studied almost-minimizers for functionals of the type first studied by Alt and Caffarelli in \cite{AC} and Alt, Caffarelli and Friedman in \cite{ACF}. 
Almost-minimization is the natural property to consider once 
the presence of noise or lower order terms 
in a problem is taken into account. In this paper we study the regularity of almost minimizers to energy functionals with variable coefficients (as opposed to \cite{davidtoroalmostminimizers}, \cite{davidengelsteintoro}, \cite{AC} and \cite{ACF} which deal only with the ``Laplacian" setting). 

The point of the present generalization is to allow anisotropic energies that depend mildly on
the point of the domain, so that in particular our classes of minimizers should be essentially
invariant by $C^{1+\alpha}$ diffeomorphisms.

The variable coefficient problem has been studied before: Caffarelli, in \cite{CafHar3}, proved regularity for solutions to a more general free boundary problem. De Queiroz and Tavares, in \cite{QT}, provided the first results for almost minimizers with variable coefficients: the authors proved regularity away from the free boundary for almost minimizers to the same functionals we consider here (they consider a slightly broader class of functionals, of which our functionals are a limiting case). 

Our work differs from that of \cite{QT} in two ways: first, our definition of almost-minimizing is, {\it a priori}, broader than that considered in \cite{QT}, \cite{davidtoroalmostminimizers} or \cite{davidengelsteintoro} (for more discussion see Section \ref{S:addmult} below). Second, and more significantly, we prove Lipschitz regularity up to, and across, the free boundary, in contrast to \cite{QT}, thus we fully generalize the results of \cite{davidtoroalmostminimizers} to the variable coefficient setting. In a forthcoming paper where we address the free boundary regularity for $(\kappa, \alpha)$-almost minimizers in the variable coefficient setting, we tackle the important issues of compactness for sequences of almost minimizers and nondegeneracy properties of almost minimizers near the free boundary. 

Besides including the notion of almost-minimizers from \cite{QT}, \cite{davidtoroalmostminimizers} or \cite{davidengelsteintoro}, our definition of almost minimizers also connects to the work of \cite{GZ}. There, the authors extend the notion of $\omega$-minimizers introduced by Anzellotti in \cite{A}, to the framework of multiple-valued functions in the sense of Almgren, and prove H\"{o}lder regularity of Dirichlet  multiple-valued $(c,\alpha)$-almost minimizers.

Almost-minimizers to functionals of Alt-Caffarelli or Alt-Caffarelli-Friedman type with variable coefficients arise naturally in measure-penalized minimization problems for Dirichlet eigenvalues of elliptic operators (e.g. the Laplace-Beltrami operator on a manifold; see \cite{LS} for a treatment of the analogous measure-constrained problem). We also want to draw attention to the interesting paper \cite{STV}, which proves (using an epiperimetric inequality) free boundary regularity for almost-minimizers of the functionals considered here, in dimension $n = 2$. Throughout that paper they need to assume {\it a priori} Lipschitz regularity on the minimizer. Our paper shows (as alluded to in their paper) that this assumption is redundant.  Note that while the class of almost-minimizers considered in \cite{STV} may seem broader than the one considered here, the two are actually equivalent (see Remark \ref{equivam} below). 

\begin{remark}\label{r:trey}
Shortly after this paper was initially posted to the ArXiv, B. Trey posted \cite{Trey}. In that paper, the author studies the regularity of shape optimizers for variable coefficient divergence form elliptic operators (e.g. the domain of area one which minimizes the sum of the first $k$ Dirichlet eigenvalues of a given operator). 

In particular, Trey (also adapting the approach of \cite{davidtoroalmostminimizers}), proves very similar results to ours but for vectorial ``quasi-minimizers" of \eqref{eqn:defjplus} with the additional property that they are solutions of a divergence form elliptic PDE with right hand side (see Theorem 1.2 in \cite{Trey}). 

We would like to emphasize that our results neither imply nor are implied by those of \cite[Theorem 1.2]{Trey}, due to the different notions of ``minimization" used and to the presence of an underlying PDE in \cite{Trey} (solutions of which were known to satisfy Theorem \ref{almostmonotonicity}). 
\end{remark}

The structure of the paper is as follows. 
In Section \ref{S:prelim} we introduce our notion of $(\kappa,\alpha)$-almost-minimizer, recalling the one used in \cite{davidtoroalmostminimizers}, \cite{davidengelsteintoro} and \cite{QT}. 
In Section \ref{S:coordinate} we address basic facts regarding the change of coordinates that will be used throughout the paper; in Section \ref{S:addmult} we address the connection between the 
``multiplicative" almost-minimizers used in \cite{davidtoroalmostminimizers} and \cite{QT} and 
the ``additive" almost-minimizers used here. 
In Section \ref{S:continuity} we prove the 
continuity of almost-minimizers; in Section \ref{S:C1beta} we prove the 
$C^{1,\beta}$ regularity of almost minimizers in $\{u>0\}$ and $\{u<0\}$. In Section \ref{S:technical} we prove the bulk of the technical results needed to obtain local Lipschitz regularity for both the one phase and two-phase problems. In Section \ref{S:onephaselip} we prove the 
local Lipschitz continuity of almost minimizers of the one-phase problem. In Section \ref{S:am} we establish an analogue of the Alt-Caffarelli-Friedman monotonicity formula for variable coefficient almost-minimizers. Finally, in Section \ref{sec:twophaselipschitz} we prove the 
local Lipschitz continuity for two-phase almost minimizers.

\section{Preliminaries}\label{S:prelim}

We consider a bounded domain $\O\subset\R^n$, $n \geq 2$, 
and study the regularity of the free boundary of almost minimizers of the functional 
\begin{equation}\label{eqn:defj}
J(u)=\int_\O \left\langle A(x)\nabla u(x), \nabla u(x)\right\rangle
+q^2_+(x)\chi_{\{u>0\}}(x) +q^2_-(x)\chi_{\{u<0\}}(x), \end{equation}
where $q_+, q_- \in L^\infty(\O)$ are bounded real valued functions and $A \in C^{0,\alpha}(\O; \R^{n\times n})$ is a H\"older continuous function with values in symmetric, uniformly positive definite matrices. Let $0 < \lambda \leq \Lambda < \infty$ be such that $\lambda |\xi|^2 \leq A(x)\xi\cdot \xi \leq \Lambda |\xi|^2$ for all $x\in \O$. 

We will also consider the situation where $ u \geq 0$ and $q_- \equiv 0$, and
\begin{equation} \label{eqn:defjplus}
J^+(u)=\int_\O\left\langle A\nabla u,\nabla u\right\rangle+q^2_+(x)\chi_{\{u>0\}},
\end{equation}
where $q_+$ and $A$ 
are as above. 

We do not need any boundedness or regularity assumption on $\O$,
because our results will be local and so we do not need to define a trace on $\p \O$. 
Also, $q_-$ is not needed when we consider $J^+$, and then we may 
assume that it is identically zero. 

\begin{defn}[Definition 1 of almost minimizers, with balls]\label{d1}  
Set
\begin{equation} \label{eqn:Kloc}
K_{\loc}(\O) = \left\{u\in L^1_{\loc}(\O) \, ; \nabla u\in L^2(B(x,r))
\mbox{ for every ball $B(x,r)$ such that } \overline B(x,r) \subset \O \right\},
\end{equation}
\begin{equation} \label{eqn:Kplusloc}
K_{\loc}^+(\O) = \left\{u\in K_{\loc}(\O) \, ; u(x) \geq 0 
\mbox{ almost everywhere on } \O\right\}, 
\end{equation}
and let constants $\kappa \in (0,+\infty)$ and $\alpha \in (0,1]$
be given.

We say that $u$ is a $(\kappa, \alpha)$-almost minimizer for $J^+_B$ in $\O$ 
if $u\in K_{\loc}^+(\O)$ and
\begin{equation}\label{eqn:amjplus}
J^+_{B,x,r}(u)\le J^+_{B,x,r}(v) + \kappa r^{n+\alpha}
\end{equation}
for every ball $B(x,r)$ such that $\overline B(x,r) \subset\O$ and every  
$v\in L^1(B(x,r))$ such that  $\nabla v\in L^2(B(x,r))$ 
and $v=u$ on $\p B(x,r)$, where
\begin{equation}\label{eqn:jplusonball}
J^+_{B,x,r}(v)=\int_{B(x,r)}\left\langle A\nabla v,\nabla v\right\rangle+q^2_+ \, \chi_{\{v>0\}}.
\end{equation}

Similarly, we say that $u$ is a $(\kappa, \alpha)$-almost minimizer for $J_B$ in $\Omega$ if $u \in K_{\mathrm{loc}}(\Omega)$ and \begin{equation}\label{eqn:amj}
J_{B,x,r}(u)\le J_{B,x,r}(v) + \kappa r^{n+\alpha}
\end{equation}
for every ball $B(x,r)$ such that $\overline B(x,r) \subset\O$ and every  
$v\in L^1(B(x,r))$ such that  $\nabla v\in L^2(B(x,r))$ 
and $v=u$ on $\p B(x,r)$, where
\begin{equation}\label{eqn:jonball}
J_{B,x,r}(v)=\int_{B(x,r)}\left\langle A\nabla v,\nabla v\right\rangle+q^2_+ \, \chi_{\{v>0\}} + q^2_-\, \chi_{\{v < 0\}}.
\end{equation}
\end{defn}

When we say $v=u$ on $\p B(x,r)$, we really mean that their traces coincide.
Equivalently we could extend $v$ by setting $v=u$ on $\Omega \sm B(x,r)$ and require that
$v \in K_{\loc}(\O)$. This is discussed in detail in \cite{davidtoroalmostminimizers}.

We note that this definition differs from the one found in \cite{davidtoroalmostminimizers} (or \cite{QT}),
even when $A$ is the identity matrix; we will address this 
in more detail in Section \ref{S:addmult}. For now, let us only comment that the definition given by \eqref{eqn:amj} is 
more general than that of \cite{davidtoroalmostminimizers}. 

When working with variable coefficients, it is also convenient to work with a definition of almost minimizers 
that considers ellipsoids instead of balls. For this effect, we define 

\begin{equation}\label{tt-1}
T_{x}(y)=A^{-1/2}(x)(y-x)+x, \qquad T_{x}^{-1}(y)=A^{1/2}(x)(y-x)+x,  \qquad E_x(x,r)=T_x^{-1}(B(x,r)).
\end{equation}
Note that
\begin{equation}\label{tt-2}
E_x(x,r)\subset B(x,\Lambda^{1/2}r) \qquad\hbox{ and }\qquad B(x,r)\subset E_x(x,\lambda^{-1/2}r).
\end{equation}

\begin{defn}[Definition 2 of almost minimizers, with ellipsoids]\label{d2} 
Let
\begin{equation} \label{eqn:Kloc2}
K_{\loc}(\O,E) = \left\{u\in L^1_{\loc}(\O) \, ; \nabla u\in L^2(E_x(x,r))
\mbox{ for } \overline E_x(x,r) \subset\O \right\} 
\end{equation}
and
\begin{equation} \label{eqn:Kplusloc2}
K_{\loc}^+(\O,E) = \left\{u\in K_{\loc}(\O,E) \, ; u(x) \geq 0 
\mbox{ almost everywhere on } \O\right\}.
\end{equation}

We say that $u$ is a $(\kappa, \alpha)$-almost minimizer for $J^+_E$ in $\O$ if
$u\in K_{\loc}^+(\O,E)$ and
\begin{equation}\label{eqn:amjplus2}
J^+_{E,x,r}(u)\le J^+_{E,x,r}(v) + \kappa r^{n+\alpha}
\end{equation}
for every ellipsoid $E_x(x,r)$ such that $\overline{E}_x(x,r) \subset\O$ and every  
$v\in L^1(E_x(x,r))$ such that  $\nabla v\in L^2(E_x(x,r))$ 
and $v=u$ on $\p E_x(x,r)$, where
\begin{equation}\label{eqn:jplusonball2}
J^+_{E,x,r}(v)=\int_{E_x(x,r)}\left\langle A\nabla v,\nabla v\right\rangle+q^2_+ \, \chi_{\{v>0\}}.
\end{equation}

Similarly, we say that $u$ is a $(\kappa, \alpha)$-almost minimizer for $J_E$ in $\Omega$ if $u \in K_{\mathrm{loc}}(\Omega,E)$ and 
\begin{equation}\label{eqn:amj2}
J_{E,x,r}(u)\le J_{E,x,r}(v) + \kappa r^{n+\alpha}
\end{equation}
for every ellipsoid $E_x(x,r)$ such that $\overline{E}_x(x,r) \subset\O$ and every  
$v\in L^1(E_x(x,r))$ such that  $\nabla v\in L^2(E_x(x,r))$ 
and $v=u$ on $\p E_x(x,r)$, where
\begin{equation}\label{eqn:jonballg} 
J_{E,x,r}(v)=\int_{E_x(x,r)}\left\langle A\nabla v,\nabla v\right\rangle+q^2_+ \, \chi_{\{v>0\}} + q^2_-\, \chi_{\{v < 0\}}.
\end{equation}
\end{defn}

Notice that when $A=I$ (the identity matrix), both definitions coincide. Moreover, for a general matrix $A$, 
if $u$ is a $(\kappa, \alpha)$-almost minimizer for $J_B$ in $\Omega$ according to Definition \ref{d1}, 
then it satisfies \eqref{eqn:amjplus2} in Definition \ref{d2} (with constant $\Lambda^{(n+\alpha)/2}\kappa$ and exponent $\alpha$) whenever $x$ and $r$ are such that $\overline{B}(x,\Lambda^{1/2}r)\subset \Omega$.

Similarly, if $u$ is a $(\kappa, \alpha)$-almost minimizer for $J_E$ in $\Omega$ according to Definition \ref{d2}, then it satisfies \eqref{eqn:amjplus} in Definition \ref{d1} (with constant $\lambda^{-(n+\alpha)1/2}\kappa$ and exponent $\alpha$) whenever $x$ and $r$ are such that $\overline{B}(x,\Lambda^{1/2}\lambda^{-1/2}r)\subset \Omega$.

Given that we are mostly interested in the regularity of almost-minimizeres away from $\partial \Omega$ these definitions are essentially equivalent. Bearing this in mind, we will work with almost minimizers according to Definition \ref{d2}, recalling that such functions satisfy \eqref{eqn:amjplus} when 
$\overline{B}(x,\Lambda^{1/2}\lambda^{-1/2}r)\subset \Omega$. 
We will most often not write ``$(\kappa, \alpha)$-almost minimizer", but only ``almost minimizer", and we will drop the subscripts $B$ and $E$ from the energy functional.\\

\noindent{\bf Notation:} Throughout the paper we will write $B(x,r)=\{y\in\R^n \ : \ |y-x|<r\}$ 
and $\partial B(x,r)=\{y\in\R^n \ : \ |y-x|=r\}$. We will consider $A \in C^{0,\alpha}(\O; \R^{n\times n})$ a H\"older continuous function with values in symmetric, uniformly positive definite matrices, 
and $0 < \lambda \leq \Lambda < \infty$ such that $\lambda |\xi|^2 \leq A(x)\xi\cdot \xi \leq \Lambda |\xi|^2$ for all $x\in \O$. Additionally, $q_\pm \in L^\infty(\O)$ will be bounded real valued functions. 
We will also frequently refer to
\begin{equation}\label{defTE}
T_{x}(y)=A^{-1/2}(x)(y-x)+x, \,\,\,\, T_{x}^{-1}(y)=A^{1/2}(x)(y-x)+x,  \,\,\,\, E_x(x,r)=T_x^{-1}(B(x,r)).
\end{equation}

Moreover, we will write 
\begin{equation}\label{defux}
u_x(y)=u(T_x^{-1}(y)),  \,\,\, 
(q_x)_{\pm}(y)=q_{\pm}(T_x^{-1}(y)), 
\,\,\,  
A_x(y)=A^{-1/2}(x)A(T_x^{-1}(y))A^{-1/2}(x).
\end{equation}

Notice that $T_x(x)=x$ and $A_x(x)=I$.

\subsection{Coordinate changes}\label{S:coordinate}

Compared to \cite{davidtoroalmostminimizers} and \cite{davidengelsteintoro}, 
our proofs will use two new ingredients: the good invariance properties of our notion with respect to 
bijective affine transformations, and the fact that the slow variations of $A$ allow us to make approximations
by freezing the coefficients. We take care of the first part in this subsection.

Many of our proofs will use the affine mapping $T_{x}$ to transform our almost minimizer $u$ into
another one $u_x$, which corresponds to a new matrix function $A_x(y)$ that coincides with the identity at $x$.
In this subsection we check that our notion of almost minimizer behaves well under bijective affine 
transformations. Our second definition, with ellipsoids, is more adapted to this.

\begin{lemma}\label{lem: changingvariables}
Let $u$ be a $(\kappa,\alpha)$-almost minimizer for $J_E$ (or $J^+_E$) in $\Omega \subset \mathbb R^n$.
Let $T: \R^n \to \R^n$ be an injective affine mapping, and denote by $S$ the 
linear map corresponding to $T$ (i.e $Tx=Sx+ z_0$ for some $z_0\in \R^n$). Also let $0 < a \leq b <+\infty$ be such that 
$a|\xi| \leq |S\xi| \leq b|\xi|$ for $\xi \in \R^n$. Then define functions
$u_T$, $q_{T,+}$, $q_{T,-}$ on $\Omega_T = T(\Omega)$ by
\begin{equation}\label{ux0Sdefbis} 
u_{T}(y) = u(T^{-1}(y)) \ \text{ and }\,  q_{T,\pm}(y) = q_\pm(T^{-1}(y)) \ \text{ for } y\in \Omega_T,
\end{equation} 
and a matrix-valued function $A_T$ by 
\begin{equation}\label{Ax0bis}
A_{T}(y) = S A(T^{-1}y)S^t \ \text{ for } y\in \Omega_T,
\end{equation}
where $S^t$ is the transposed matrix of $S$.

Then $u_T$ is a $(\tilde{\kappa}, \alpha)$-almost minimizer of $J_{E,T}$ (or $J^+_{E,T}$) 
in $\Omega_T$, according to Definition \ref{d2}, where $J_{E,S}$ (or $J^+_{E,S}$) 
is defined in terms of $A_{T}$ and the $ q_{T,\pm}$, i.e.,
\begin{equation}\label{Jx0}
 J_{E,S}(v) = \int \left\langle A_{T}(y)\nabla v(y), \nabla v(y)\right\rangle 
 + q_{T,+}^2(y)\chi_{\{v > 0\}}(y) + q_{T,-}^2(y)\chi_{\{v < 0\}}(y)\ dy,
  \end{equation}
and $\tilde{\kappa} =  \kappa |\det T|$.
\end{lemma}

 \begin{remark}\label{rem: changingvariables} 
 Lemma \ref{lem: changingvariables} says that under an affine change of variables, almost minimizers are transformed to almost minimizers for a modified functional.  
 Its proof will also show why our second definition of almost minimizers is natural. 
 But it will be applied almost exclusively in the following circumstances: we pick
 $x \in  \Omega$, and we take $S = A^{-1/2}(x)$. In this case,
 $T(y) = x + S(y-x)$, we recognize the affine mapping $T_{x}$ from \eqref{defTE}, and
 then $u_T = u_x$ and $A_T(y) = A_x(y)$ (from \eqref{defux}).
 The advantage is that $A_T(x_0) = I$ and we can use simpler competitors.
 \end{remark}

\begin{proof}
We do the proof for $J^+$; the argument for $J$ would be the same.
Let $u$, $T$, and $u_T$ be as in the statement, then let $E_T$ be the corresponding ellipsoid and 
$v_T\in L^1(E_T)$ define a competitor for $u_T$ as in Definition \ref{d2};
thus $\overline E_T \subset \Omega_T$, $\nabla v_T\in L^2(E_T)$,  and $v_T=u_T$ on $\p E_T$.
We want to use $v_T$ to define a competitor $v$ for $u$, and naturally we take 
$v(y) = v_T(T(y))$ for $y\in E = T^{-1}(E_T)$. 

Notice that $\overline T^{-1}(E_T) \subset \Omega$ because $\overline E_T \subset \Omega_T = T(\Omega)$.
Moreover, \eqref{defTE} and \eqref{Ax0bis} say that $E_T - x$ is the image of $B(0,r)$ by
the linear mapping $A_T^{1/2} = SA^{1/2}(x')$, where $x'=T^{-1}(x)$. Then 
$T^{-1}(E_T) - x'$ is the image of $B(0,r)$ by $S^{-1}SA^{1/2}(x') = A^{1/2}(x')$.
In other words, $E = T^{-1}(E_T)$ is the ellipsoid associated to $x'$ and our initial function $A$, 
as in \eqref{defTE}, and we can apply Definition \ref{d2} to $v$. It is clear that $v=u$ on $\p E$,
and $\nabla v\in L^2(E)$ because the differential is $Dv(z) = Dv(T(z))S$ (and you transpose to get the gradients).
Now we compute, setting $y=T(z)$ and eventually changing variables,
\begin{align}\label{cvg}
J^+_{E,x',r}(v) &= \int_{E} \left\langle A\nabla v,\nabla v\right\rangle + q^2_+ \, \chi_{\{v>0\}} dz
= \int_{E} \left\langle A(z)S^t \nabla v(y), S^t\nabla v(y)\right\rangle + q^2_+ \, \chi_{\{v>0\}}(z) dz
\nonumber\\
&=\int_{E} \left\langle A_T(y)\nabla v(y), \nabla v(y)\right\rangle + q^2_+ \, \chi_{\{v_T>0\}}(y) dz
\nonumber\\
&=|\det(T)| \int_{E_T} \left\langle A_T(y)\nabla v(y), \nabla v(y)\right\rangle + q^2_+ \, \chi_{\{v_T>0\}}(y) dy,
\end{align}
which is the analogue (call it $J_{E_T}(v_T)$)  of $J^+_{E,x,r}$ for $v_T$ on $E_T$. 
We have a similar formula for 
$J^+_{E,x',r}(u)$, and since $J^+_{E,x',r}(u) \leq J^+_{E,x',r}(v) + \kappa r^{n+\alpha}$
by \eqref{eqn:amj2}, we get that $J^+_{E_T}(u_T) \leq J^+_{E_T}(v_T) + |\det(T)| \kappa r^{n+\alpha}$.
Lemma \ref{lem: changingvariables} follows.
\end{proof}

 In the analysis below we are working entirely locally within $\Omega$ and are unconcerned with the precise dependence of our regularity on $\kappa$ and $\alpha$. Therefore, we will sometimes make the {\it a priori} assumption (justified by the analysis above) that for a given point $x_0 \in \Omega$ we have $A(x_0) = I$. When it is necessary to compare different points in $\Omega$, we will explicitly use the rescaled functions defined above.

Whenever we write $C$, we mean a constant (which might change from line to line) that depends on $n$, $\lambda$, $\Lambda$,  $||q_{\pm}||_{L^{\infty}}$, $\alpha$ and on upper bounds for $||A||_{C^{0,\alpha}}$, and $\kappa$.

\subsection{``Additive" Almost-Minimizers}\label{S:addmult}

 Let us now address the differences between our definition of almost minimizers, with 
\eqref{eqn:amj} or  \eqref{eqn:amj2}, and the definition of an almost-minimizer in \cite{davidtoroalmostminimizers}. 
Recall that when $A=I$, 
being an almost minimizer for $J_E$ is equivalent to being an almost minimizer for $J_B$, and that in  \cite{davidtoroalmostminimizers} (with $A=I$) $u$ was an almost-minimizer for $J_E$ if, instead of satisfying \eqref{eqn:amj2} for all admissible $v$, 
it satisfied 
\begin{equation}\label{eqn:amjmult} 
J_{E,x,r}(u) \leq (1+\kappa r^\alpha) J_{E,x,r}(v), 
\end{equation} 
(and similarly for $J^+_E$).  
Here we consider variable $A$ and stick to $J_E$ (but $J_B$ would work the same way).
Let almost minimizers in the sense of \eqref{eqn:amj2} be {\bf additive almost minimizers}, 
whereas almost-minimizers in the sense of \eqref{eqn:amjmult} are {\bf multiplicative} almost-minimizers. 
Our goal is to prove results for additive minimizers, first showing that multiplicative almost minimizers are also additive almost minimizers. To obtain this result we first need to show that multiplicative almost minimizers, in the variable coefficient setting, obey a certain decay property. This will be done in the next Lemma. With this result in hand, we will then show that every multiplicative almost minimizer is actually an additive almost minimizer,
therefore reducing our analysis to the case of additive minimizers. 
 
 \begin{lemma}\label{L:logdecay} 
 Let $u$ be a multiplicative almost minimizer for $J_E$ in $\Omega$. 
 Then there exists a constant $C>0$ such that if $x\in\Omega$ and $r>0$ are such that 
 $\overline E_x(x,r) \subset \Omega$, 
 then for $0<s\le r$,
 \begin{equation}\label{multiplicativeam}
 \left(\fint_{E_x(x,s)}|\nabla u|^2\right)^{1/2}\le C\left(\fint_{E_x(x,r)}|\nabla u|^2\right)^{1/2}+C\log(r/s).
 \end{equation}
\end{lemma}
 
 \begin{proof} Our assumption that $\overline E(x,r) \subset \Omega$ allows us to define  
 $u_x$ as in \eqref{defux}. Note that $u_x\in H^{1/2}(\partial B(x,s))$.
  % and will allow us to use the almost minimality of $u$ below.
Denote by $(u_{x})_s^*$ 
 the function in $L^1(B(x,s))$ with $\nabla (u_{x})_s^*\in L^2(B(x,s))$, trace 
$u_{x}$ on $\partial B(x,s) = T_x(\partial E_x(x,s))$, and 
which minimizes the Dirichlet energy on $B(x,s)$. 
%The existence and uniqueness of such a function follow from the fact that we start from the trace of $u_x$
%in the Sobolev space $H^{1/2}(\partial B(x,s))$, which itself has an extension to $B(x,s)$
%with one derivative in $L^2$ (in fact, $u_x$ itself), and from the convexity of the Dirichlet energy. 
%When $u|_{\partial E_x(x,s)}$ and $(u_x)|_{\partial B(x,s)}$ are regular, 
Note that $u_s^*$ is the harmonic 
extension of $(u_x)|_{\partial B(x,s)}$.  
The minimality of $u_s^*$ implies that for any $t\in\R$,
 \[
 \int_{B(x,s)}|\nabla(u_{x})_s^*|^2\le \int_{B(x,s)}|\nabla ((u_{x})_s^*+t(u_{x}-(u_{x})_s^*))|^2.
 \]
 Expanding near $t=0$ we obtain $\int_{B(x,s)}\langle \nabla u_{x}-\nabla (u_{x})_s^*,\nabla (u_{x})_s^*\rangle=0$, hence
 \begin{equation}\label{2.3mult}
 \int_{B(x,s)}|\nabla (u_{x})_s^*|^2=\int_{B(x,s)}\langle \nabla u_{x},\nabla (u_{x})_s^*\rangle.
 \end{equation}
 Since $(u_x)_s^*\circ T_x\in L^2(E_x(x,s))$ and its trace is equal to $u$ on $\partial E_x(x,s)$, 
 \eqref{defux}, the almost minimality of $u$ and the same computation as in \eqref{cvg} 
 (in fact, we are in the situation of Remark \ref{rem: changingvariables}) yield
 \[
  \begin{aligned}
  \text{det}&A^{1/2}(x) 
  \int_{B(x,s)}\langle A_x\nabla u_x,\nabla u_x\rangle+(q_x)_+^2\chi_{\{u_x>0\}}+(q_x)_-^2\chi_{\{u_x<0\}}\\
  &= \int_{E_x(x,s)}\langle A\nabla u,\nabla u\rangle +q_+^2\chi_{\{u>0\}}+q_-^2\chi_{\{u<0\}}\\
  &\le (1+\kappa s^{\alpha})\int_{E_x(x,s)}\langle A\nabla ((u_x)_s^*\circ T_x),\nabla ((u_x)_s^*\circ T_x)\rangle +q_+^2\chi_{\{(u_x)_s^*\circ T_x>0\}}+q_-^2\chi_{\{((u_x)_s^*\circ T_x)<0\}}\\
 &=(1+\kappa s^{\alpha})\text{det}A^{1/2}(x) \int_{B(x,s)}\langle A_x\nabla (u_x)_s^*,\nabla (u_x)_s^*\rangle+(q_x)_+^2\chi_{\{(u_x)_s^*>0\}}+(q_x)_-^2\chi_{\{(u_x)_s^*<0\}}.
\end{aligned}
\]
 Consequently,
 \begin{align}\label{well}
 \int_{B(x,s)}&\langle A_x\nabla u_x,\nabla u_x\rangle+(q_x)_+^2\chi_{\{u_x>0\}}+(q_x)_-^2\chi_{\{u_x<0\}}
 \nonumber \\
 &\le(1+\kappa s^{\alpha}) \int_{B(x,s)}\langle A_x\nabla (u_x)_s^*,\nabla (u_x)_s^*\rangle+(q_x)_+^2\chi_{\{(u_x)_s^*>0\}}+(q_x)_-^2\chi_{\{(u_x)_s^*<0\}}
 \nonumber \\ 
 &\le  C s^n + (1+\kappa s^{\alpha}) \int_{B(x,s)}\langle A_x\nabla (u_x)_s^*,\nabla (u_x)_s^*\rangle 
 \end{align} 
where $C$ depends on the $||q_{\pm}||_\infty$ and an upper bound for $\kappa s^{\alpha}$.
Observe that for $z\in B(x,s)$, 
 \begin{align}\label{1.25g}
 |I-A_x(y)| &= |A_x(x)-A_x(y)| = |A^{-1/2}(x)[A(T_x^{-1}(x))-A(T_x^{-1}(y))]A^{-1/2}(x)|
 \nonumber \\&
 \leq \lambda^{-1}|x-y|^\alpha ||A||_{C^\alpha}  \leq  C s^\alpha
 \end{align} 
 by \eqref{defux} (twice). Then by \eqref{2.3mult}, \eqref{well}, \eqref{1.25g} (twice), 
 \begin{align}
 \int_{B(x,s)}&|\nabla u_x-\nabla (u_x)_s^*|^2
 =\int_{B(x,s)}|\nabla u_x|^2 - \int_{B(x,s)}|\nabla (u_x)_s^*|^2
 \nonumber\\
 &\le \int_{B(x,s)}\langle (I-A_x) 
 \nabla u_x,\nabla u_x\rangle
 +\int_{B(x,s)}\langle A_x\nabla u_x,\nabla u_x\rangle-\int_{B(x,s)}|\nabla (u_x)_s^*|^2
 \nonumber\\
 &\le Cs^{\alpha}\int_{B(x,s)}|\nabla u_x|^2-\int_{B(x,s)}|\nabla (u_x)_s^*|^2+\int_{B(x,s)}\langle A_x\nabla u_x,\nabla u_x\rangle\nonumber\\
 &\le Cs^{\alpha}\int_{B(x,s)}|\nabla u_x|^2-\int_{B(x,s)}|\nabla (u_x)_s^*|^2
+ C s^n + (1+\kappa s^{\alpha}) \int_{B(x,s)}\langle A_x\nabla (u_x)_s^*,\nabla (u_x)_s^*\rangle
\nonumber \\ 
 &\le Cs^{\alpha}\int_{B(x,s)}|\nabla u_x|^2-\int_{B(x,s)}|\nabla (u_x)_s^*|^2
 + C s^n
 +(1+\kappa s^{\alpha})(1+s^{\alpha}) 
  \int_{B(x,s)}|\nabla (u_x)_s^*|^2
 \nonumber\\
&\le Cs^{\alpha}\int_{B(x,s)}|\nabla u_x|^2+Cs^n + C s^{\alpha} \int_{B(x,s)}|\nabla (u_x)_s^*|^2
 \leq Cs^{\alpha}\int_{B(x,s)}|\nabla u_x|^2+Cs^n
\nonumber\\&\label{2.4mult} 
\end{align}
where we finished with the minimality of $(u_x)_s^*$.

Applying \eqref{2.4mult} to $s=r$ yields
 \begin{equation}\label{2.5mult}
 \int_{B(x,r)}|\nabla u_x-\nabla (u_x)_r^*|^2\le C r^{\alpha}\int_{B(x,r)}|\nabla u_x|^2 + Cr^n.
 \end{equation}
 We may now follow the computations in \cite{davidtoroalmostminimizers}, to which we refer for additional detail.
 Set
 \begin{equation}\label{2.6mult}
 \omega(u_x,x,s)=\left(\fint_{B(x,s)}|\nabla u_x|^2\right)^{1/2}.
 \end{equation}
 Since $(u_x)_r^*$ is energy minimizer, it is harmonic in $B(x,r)$, therefore $|\nabla (u_x)_r^*|^2$ is subharmonic. We obtain 
 \begin{equation}\label{2.7mult}
\fint_{B(x,s)}|\nabla (u_x)_r^*|^2 \le \fint_{B(x,r)}|\nabla (u_x)_r^*|^2.
 \end{equation}
 By the triangle inequality in $L^2$, \eqref{2.5mult}, \eqref{2.6mult} and \eqref{2.7mult}, we obtain, as in \cite{davidtoroalmostminimizers},
 \begin{align}
 \omega(u_x,x,s) 
& \le \left(1+C\left(\frac{r}{s}\right)^{n/2}r^{\alpha/2}\right)\omega(u_x,x,r)+C\left(\frac{r}{s}\right)^{n/2}.\label{2.8mult}
 \end{align}
 Setting $r_j=2^{-j}r$, for $j\ge 0$, \eqref{2.8mult} gives
 \begin{equation}\label{2.9mult}
 \omega(u_x,x,r_{j+1})\le (1+C2^{n/2}r_j^{\alpha/2})\omega(u_x,x,r_j)+C2^{n/2},
 \end{equation}
 and an iteration yields
 \begin{align}
 \omega(u_x,x,r_{j+1})&\le \omega(u_x,x,r)\prod_{l=0}^{j}(1+C2^{n/2}r_l^{\alpha/2})
 +C\sum_{l=1}^{j+1}\left(\prod_{k=l}^j (1+C2^{n/2}r_k^{\alpha/2})\right)2^{n/2}
 \nonumber\\
 &\le \omega(u_x,x,r)P+CP2^{n/2}j\le C\omega(u_x,x,r)+Cj,\label{2.10mult}
 \end{align} 
where $P=\prod_{j=0}^{\infty}(1+C2^{n/2}r_j^{\alpha/2})$ and we used the fact that $P$ is bounded, 
depending only on an upper bound for $r$. 
As in \cite{davidtoroalmostminimizers}, this implies that if $\overline E_x(x,r) \subset \Omega$,
then for $0<s\le r$,
 \begin{equation}\label{2.11mult}
 \omega(u_x,x,s)\le C\omega(u_x,x,r)+C\log(r/s).
 \end{equation}
 Since 
 \begin{equation}\label{tt-3}
 \omega(u_x,x,r)\le C\left(\fint_{E_x(x,r)}|\nabla u|^2\right)^{1/2}\hbox{ and }\left(\fint_{E_x(x,s)}|\nabla u|^2\right)^{1/2}\le C\omega(u_x,x,s),
 \end{equation} 
  we obtain, for $0<s\le r$
 \begin{equation}\label{need}
\left(\fint_{E_x(x,s)}|\nabla u|^2\right)^{1/2}\le C\left(\fint_{E_x(x,r)}|\nabla u|^2\right)^{1/2}+C\log(r/s).
 \end{equation}
  \end{proof}

\begin{lemma}\label{lem:comparemintypes}
Let $u$ be a multiplicative almost-minimizer of $J_E$ in $\Omega$ with constant $\kappa$ and exponent $\alpha$, 
and let $\wt\Omega \subset\subset \Omega$ be an open subset of $\Omega$ whose closure is a compact 
subset of $\Omega$. Then $u$ is an additive almost minimizer of $J_E$ in $\wt\Omega$, 
with exponent $\alpha/2$ and a constant $\wt{\kappa}$ that depends on the constants for $J$, $u$ and $\wt\Omega$.
\end{lemma}

\begin{proof}
Let $\Omega$, $u$, $\wt\Omega$, be as in the statement, and choose
$r_0 = \Lambda^{-1/2} \dist(\wt\Omega, \p\Omega)/2$, so small that $E_x(x,2r_0) \subset \Omega$ 
for $x\in \Omega$. We deduce from Lemma \ref{L:logdecay}, applied with $r=r_0$, that
\begin{equation}\label{log2}
\fint_{E_x(x,s)}|\nabla u|^2 \leq C + C |\log(s/r_0)|^2
\ \text{ for } 0 < s \leq r_0,
\end{equation}
where $C$ only depends on $\wt\Omega$ and $u$ through a bound for $\displaystyle{\fint_{E_x(x,r_0)}|\nabla u|^2}$
by $\displaystyle{r_0^{-n}\int_{(\wt\Omega,r_0)}|\nabla u|^2}$ where $(\wt\Omega,r_0)$, the $r_0$ neighborhood of $\wt\Omega$, is compactly contained in $\Omega$
by our choice of $r_0$.

Now let $x\in \Omega, r > 0$ be such that $\overline{E}_x(x,r)\subset \tilde{\Omega}$ and let 
$v$ be an admissible function, with $v = u$ on $\partial E_x(x,r)$; we know that
$$
J_{E,x,r}(u) \le(1+\kappa r^\alpha)J_{E,x,r}(v) 
$$
and so we just need to show that $\kappa r^\alpha J_{E,x,r}(v) \leq \wt\kappa r^{n+\alpha/2}$.
But by \eqref{eqn:jonballg} 
and \eqref{log2}
\begin{equation}\label{log3}
J_{E,x,r}(u) \leq \Lambda \int_{E_x(x,r)}|\nabla v|^2 + Cr^n ||q_+||_\infty + Cr^n ||q_-||_\infty
\leq C r^n + C r^n |\log(r/r_0)|^2 
\end{equation}
and the result follows easily; we could even have taken any given exponent $\wt \alpha < \alpha$.
\end{proof}

For the remainder of the paper we will work solely with additive almost-minimizers and refer to them simply as almost-minimizers.

\begin{remark}\label{equivam}
In \cite{STV} they consider the seemingly broader class of almost-minimizers defined by 
the inequality $$J_{x,r}(u) \leq (1+C_1r^\alpha)J_{x,r}(v) + C_2 r^{\alpha + n}.$$  In fact this definition is equivalent to our ``additive" almost-minimizers. To see this, first note that Lemmas \ref{L:logdecay} and \ref{lem:comparemintypes} hold with  
the same proofs if ``multiplicative" almost-minimizers are replaced by almost-minimizers of ``\cite{STV}-type". The only change is the presence of the (lower order) term $C_2 r^{\alpha + n}$ in \eqref{2.4mult} and 
\eqref{log3}.  
\end{remark}

\section{Continuity of Almost-Minimizers}\label{S:continuity}

Given the equivalence between almost minimizers of $J_B$ and $J_E$, we will omit the subscript. 
In this section we prove the continuity of almost minimizers for $J$ and $J^+$. 
Our arguments will follow very closely those of Theorem 2.1 in \cite{davidtoroalmostminimizers}. 
Despite this, we will prove Theorem~\ref{T:continuity} in complete detail, in order to highlight the differences in the variable coefficient setting. 

Furthermore, to 
ease notation we will refer only to $J$ in this section, with the understanding that $q_-$ might be identically zero and the 
functions we consider might be {\it a priori} non-negative.

\begin{theorem}\label{T:continuity} Almost minimizers of $J$ are continuous in $\Omega$. Moreover, if $u$ is an almost minimizer for $J$ and $\overline{B}(x_0,2r_0)\subset \Omega$ then there exists a constant $C>0$ 
such that 
for $x,y\in B(x_0,r_0)$
\begin{equation}\label{continuity}
|u(x)-u(y)|\le C|x-y|\left(1+\log\left(\frac{2r_0}{|x-y|}\right)\right).
\end{equation}
\end{theorem}
 
\begin{proof} 
Let $u$ be an almost minimizer of $J$ in $\Omega$, and let $x\in\Omega$ and $0<r<1$ be such that 
$\overline{E}_x(x,r)\subset\Omega$. Define $u_x$ as in \eqref{defux}, and then 
for $0 < s \leq r$, let 
$u_s^*\in L^1(B(x,s))$ with
$\nabla u_s^*\in L^2(B(x,s))$, be the harmonic extension of  $u_x$ 
in $ B(x,s) = T_x( E_x(x,s))$. Recall that $u_s^*$ minimizes the Dirichlet
energy among all functions that coincide $u_x$ on  $ \partial B(x,s)$.
%
%with 
%the same trace as $u_x$ on $\partial B(x,r) = T_x(\partial E_x(x,r))$, and 
%which minimizes the Dirichlet energy on $B(x,s)$. 
%This is the same function as in the first lines of the proof of Lemma \ref{L:logdecay},
%the justification of existence and uniqueness is the same, and \eqref{2.3mult} holds because
%$u_s^*$ minimizes the Dirichlet energy.

Let us assume for the moment that $A(x)=I$; this will simplify the computation, 
in particular because $E_x(x,r) = B(x,r)$ and $u_x=u$,
and then we will use Lemma \ref{lem: changingvariables} to reduce to that case.
Since $A\in C^{0,\alpha}(\Omega; \R^{n\times n})$, \eqref{2.3mult} yields 
\begin{align}
\int_{B(x,s)} |\nabla u-\nabla u_s^*|^2 &=\int_{B(x,s)}|\nabla u|^2-|\nabla u_s^*|^2\nonumber\\
&=\int_{B(x,s)}\langle (A(x)-A(y))\nabla u,\nabla u\rangle+\int_{B(x,s)}\langle A(y)\nabla u,\nabla u\rangle-\int_{B(x,s)}|\nabla u_s^*|^2\nonumber\\
&\le Cs^{\alpha}\int_{B(x,s)}|\nabla u|^2-\int_{B(x,s)}|\nabla u_s^*|^2
+\int_{B(x,s)}\langle A\nabla u,\nabla u\rangle,
\label{gradientcomparison}
\end{align}
where we used \eqref{1.25g} to control $|A(x)-A(y)|$.
Since $u$ is an almost minimizer and $q_{\pm}\in L^{\infty}$, 
\begin{align}
\int_{B(x,s)}\langle A\nabla u,\nabla u\rangle
&= J_{x,s}(u)-\int_{B(x,s)}q_+^2\chi_{\{u>0\}} +q_-^2\chi_{\{u<0\}} \leq J_{x,s}(u)
\le J_{x,s}(u_s^*)+\kappa s^{n+\alpha}
\nonumber\\
&\le \int_{B(x,s)}\langle A\nabla u_s^*,\nabla u_s^*\rangle +\kappa s^{n+\alpha}+Cs^n
\nonumber\\
&= \int_{B(x,s)}|\nabla u_s^*|^2
+\int_{B(x,s)}\langle (A(y)-A(x))\nabla u_s^*,\nabla u_s^*\rangle 
+\kappa s^{n+\alpha}+Cs^n\nonumber\\
&\leq (1+Cs^\alpha) \int_{B(x,s)}|\nabla u_s^*|^2 + \kappa s^{n+\alpha}+Cs^n,
\label{gradientcomparison2a}
\end{align}
by \eqref{1.25g} again. Hence by \eqref{gradientcomparison} and since  
$\int_{B(x,s)}|\nabla u_s^*|^2 \leq \int_{B(x,s)}|\nabla u_s|^2$ by definition of $u_s^*$,
\begin{align}
\int_{B(x,s)} |\nabla u-\nabla u_s^*|^2  
&\le Cs^{\alpha} \int_{B(x,s)}|\nabla u|^2+Cs^{\alpha}\int_{B(x,s)}|\nabla u_s^*|^2
+\kappa s^{n+\alpha}+Cs^n
\nonumber\\
&\leq Cs^{\alpha} \int_{B(x,s)}|\nabla u|^2 +Cs^n.
\label{gradientcomparison3}
\end{align}
In particular, when applied to $s=r$,
\begin{equation}\label{gradientcomparisonr}
\int_{B(x,r)}|\nabla u-\nabla u_r^*|^2\le C r^{\alpha}\int_{B(x,r)}|\nabla u|^2+Cr^n.
\end{equation}
For $s>0$ such that $B(x,s)\subset \Omega$, define as in \eqref{2.6mult}
\begin{equation}\label{w}
\omega(u,x,s):=\left( \fint_{B(x,s)} |\nabla u|^2\right)^{1/2}. 
\end{equation}
Since $u_r^*$ is harmonic in $B(x,r)$, $|\nabla u_r^*|^2$ is subharmonic and
for $s\le r$,
\begin{equation}\label{subharmonic}
\left(\fint_{B(x,s)}|\nabla u_r^*|^2\right)^{1/2}\le \left(\fint_{B(x,r)}|\nabla u_r^*|^2\right)^{1/2}\le \left(\fint_{B(x,r)}|\nabla u|^2\right)^{1/2}
\end{equation}

Combining \eqref{gradientcomparisonr}, \eqref{w} and \eqref{subharmonic} as in (2.10) in \cite{davidtoroalmostminimizers} and \eqref{2.10mult}, we obtain, for some $C>0$,
\begin{eqnarray}\label{firstw}
\omega(u,x,s)&\le & \left(\fint_{B(x,s)}|\nabla u- \nabla u_r^*|^2\right)^{1/2} + \left(\fint_{B(x,s)}|\nabla u_r^*|^2\right)^{1/2}\nonumber\\
&\le&\left(1+C\left(\frac{r}{s}\right)^{n/2}r^{\alpha/2}\right)\omega(u,x,r)+C\left(\frac{r}{s}\right)^{n/2}
\end{eqnarray}

Setting $r_j=2^{-j}r$ for $j\ge 0$, \eqref{firstw} implies
\[
\omega(u,x,r_{j+1})\le \left(1+C2^{n/2}r_j^{\alpha/2}\right)\omega(u,x,r_j) + C2^{n/2}.
\]
Iterating this as in (2.10) of \cite{davidtoroalmostminimizers} we obtain for $r_j=2^{-j}r$ for $j\ge 0$ that
\begin{equation}\label{iterationw}
\begin{aligned}
\omega(u,x,r_{j+1})
&\le P\omega(u,x,r)+CPj\le C\omega(u,x,r)+Cj,
\end{aligned}
\end{equation}
where $P=\prod_{j=0}^{\infty}(1+C2^{n/2}r_j^{\alpha/2})$ can be bounded depending on an upper bound for $r$.

As in (2.11) in \cite{davidtoroalmostminimizers}, this implies that if $\overline{B}(x,r)\subset \Omega$ and $A(x)=I$, then for $0<s\le r$,
\begin{equation}\label{logboundw}
\omega(u,x,s)\le C\omega(u,x,r)+C\log(r/s),
\end{equation}
where $C$ also depends on an upper bound for $r$.

\smallskip 
Now we use this to control the variations of $u$ near $x$
let $u_j=\fint_{B(x,r_j)}u$. The 
Poincar\'{e} inequality and \eqref{iterationw} yield
\begin{equation}\label{pcj}
\left(\fint_{B(x,r_j)}|u-u_j|^2\right)^{1/2}\le C r_j \omega(u,x,r_j)\le Cr_j\omega(u,x,r)+Cjr_j.
\end{equation}
If, in addition to the assumptions above, $x$ is a Lebesgue point of $u$, then
$u(x)=\lim\limits_{l\rightarrow\infty}u_l$ and we obtain, as in (2.13) from \cite{davidtoroalmostminimizers},
\begin{equation}\label{uaverage}
|u(x)-u_j| \le Cr_j(\omega(u,x,r)+j+1).
\end{equation}

\smallskip 
We may now return to the general case when $\overline E_x(x,r) \subset \Omega$
but maybe $A(x) \neq I$. In this case, Lemma \ref{lem: changingvariables}
and Remark \ref{rem: changingvariables} say that $u_x$ is an almost minimizer in the
domain $\Omega_x = T_x(\Omega)$, with the functional $J_x$ associated
to $A_x$ defined by \eqref{defux}, the same exponent $\alpha$ and 
$\wt\kappa = \det A(x)^{-1/2}\kappa$. This is good, because we can apply the argument above to
$u_x$ in $B(x,r) = T_x(E_x(x,r))$ as $A_x(x)=I$.
In particular we obtain that
\begin{equation}\label{tt-4}
\omega(u_x,x,s)\le C\left(\frac{r}{s}\right)^{n/2}r^\alpha \omega(u_x,x,r) +  C\left(\frac{r}{s}\right)^{n/2}
+\left(\fint_{B(x,s)}|\nabla (u_x)_r^*|^2\right)^{1/2},
\end{equation}
where $(u_x)_r^*$ is the harmonic extension of $u_x$ to $B(x,r)$.
 Since $u(x) = u_x(x)$ by \eqref{defTE}, we get that
\begin{equation}\label{uaverage2} 
|u(x)-u_{x,j}| = |u_x(x)-u_{x,j}| \le Cr_j(\omega(u_x,x,r)+j+1),
\end{equation}
where $u_{j,x}=\fint_{B(x,r_j)}u_x$, provided that $x$ is a Lebesgue point for $u_x$
(or, equivalently for $u$). 

\smallskip
For the continuity of $u$, we intend to apply this to Lebesgue points $x, y$ for $u$,
choose a correct $j$, and compare $u_{j,x}$ to $u_{j,y}$. This is possible if
$E_x(x,r_j) = T_x^{-1}(B(x,r_j))$ and $E_y(y,r_j) = T_x^{-1}(B(y,r_j))$ have a large intersection, 
so we need to pay attention to the size of balls.

Let $x_0\in \Omega$ and $r_0>0$ such that $\overline{B}(x_0,2r_0)\subset \Omega$ 
be given, and then let $x,y \in B(x_0,r_0)$ be given. Set $r = \Lambda^{-1/2} r_0$; this way
we are sure that $E_{x}(x,r) = T_x^{-1}(B(x,r)) \subset B(x, \Lambda^{1/2} r) \subset \overline{B}(x_0,2r_0)$
(see \eqref{defTE}), and since $u_x(y) = u(T_x^{-1}(u))$ by \eqref{defux},
\begin{align} \label{2.14g}
\omega(u_x,x,r)&:=\left(\fint_{B(x,r)}|\nabla u_x|^2\right)^{1/2} =\left(\fint_{E_x(x,r)}\langle A(x)\nabla u,\nabla u\rangle\right)^{1/2} \nonumber\\
& \le C\left(\fint_{B(x,\Lambda^{1/2}r)}|\nabla u|^2\right)^{1/2}
\leq C \left(\fint_{\overline B(x,2r_0)}|\nabla u|^2\right)^{1/2},
\end{align}
and we have a similar estimate for $\omega(u_y,y,r)$. Next assume that 
$|x-y| \leq \lambda^{1/2} r = \lambda^{1/2} \Lambda^{-1/2} r_0$. If this does not happen,
we need to take intermediate points and apply the estimates below to a string of such points.
Let $j$ be the largest integer such that $|x-y| \leq \lambda^{1/2} r_j$; we just made sure that
$j \geq 0$. Since $r_j=2^{-j}r$ then $|x-y| \geq \lambda^{1/2} r_j/2$.
Now $E_x(x,r_j) = T_x^{-1}(B(x,r_j))$ contains $B(x,\lambda^{1/2} r_j)$
and similarly $E_y(y,r_j)$ contains $B(y,\lambda^{1/2} r_j)$. Thus both sets contain the ball $B_{xy}$
centered at $(x+y)/2$ and with radius $\lambda^{1/2} r_j/2$, because $|x-y| \leq \lambda^{1/2} r_j$.
Set $m = \fint_{B_{xy}}u$; then
\begin{align}
|m-u_{j,x}| &\leq \fint_{B_{xy}}|u-u_{j,x}| \leq C \fint_{E_x(x,r_j)}|u-u_{j,x}| 
= C \fint_{B(x,r_j)}|u_x-u_{j,x}| 
\nonumber\\
&\leq C r_j \omega(u_x,x,r) + C j r_j
\leq C \left(\fint_{\overline B(x,2r_0)}|\nabla u|^2\right)^{1/2} r_j + C j r_j
\label{averageux}
\end{align}
because $B \subset E_x(x,r_j)$, by the change of variable suggested by \eqref{defTE} and \eqref{defux},
then by the Poincar\'e estimate \eqref{pcj} and \eqref{2.14g}.

We have a similar estimate for $|m-u_{j,y}|$, we compare them, and then use \eqref{uaverage2} to obtain
\begin{align}\label{2.16g}
|u(x)-u(y)| &\leq C r_j \left\{ \left(\fint_{\overline B(x,2r_0)}|\nabla u|^2\right)^{1/2} + j\right\}
\nonumber\\
&\leq C |x-y| \left\{ \left(\fint_{\overline B(x,2r_0)}|\nabla u|^2\right)^{1/2} 
+ \log\left(\frac{r_0}{|x-y|}\right)\right\}
\end{align}
for Lebesgue points $x,y \in B(x_0,r_0)$ such that $|x-y| \leq \lambda^{1/2} \Lambda^{-1/2} r_0$, 
and where $C$ depends on $\kappa, ||q_{\pm}||_{L^{\infty}(\Omega)}, \alpha, n$, an upper bound on $r$ and the $C^{0,\alpha}$ norm of $A$. By possibly modifying $u$ on a set
of measure 0, we get a continuous function that satisfies \eqref{continuity}.
\end{proof}

Here is a simple consequence of Theorem \ref{T:continuity}.
\begin{corollary}\label{C:2.1}
If $u$ is an almost minimizer for $J$, then for each compact $K\subset \Omega$, there exists a constant $C_K>0$ such that for $x,y\in K$,
\begin{equation}\label{2.21}
|u(x)-u(y)|\le C_K|x-y|\left(1+\left|\log\frac{1}{|x-y|}\right|\right).
\end{equation}
\end{corollary}

\section{Almost minimizers are $C^{1,\beta}$ in $\{u>0\}$ and in $\{u<0\}$}{\label{S:C1beta}
We first prove Lipschitz bounds away from the free boundary. Note that since $u$ is continuous,
$\{u>0\}$ and $\{u<0\}$ are open sets.

\begin{theorem}\label{T:Lippositive} Let $u$ be an almost minimizer for $J$ (or $J^+$) 
in $\Omega$. 
Then $u$ is locally Lipschitz in $\{u>0\}$ and in $\{u<0\}$.
\end{theorem}

\begin{proof}
We show the result for almost minimizers of $J$ in $\{u>0\}$, but the proof 
applies to the other cases. First 
let $x\in \{u>0\}$ be such that $A(x)=I$ and take $r>0$ 
such 
that $\overline{B}(x,2\Lambda^{1/2}\lambda^{-1/2}r)\subset \{u>0\}$.  
We start as in the proof of Lemma~\ref{L:logdecay}. 
Denote with $u_r^*$ the function with the same trace as $u$ on $\partial B(x,r)$ and which minimizes 
the Dirichlet energy under this constraint. Since $u$ is an almost minimizer we have
\begin{equation}\label{Jbound}
J_{B,x,r}(u)\le J_{B,x,r}(u_r^*)+ 
\kappa r^{n+\alpha}.
\end{equation}
Since $u>0$ in $\overline{B}(x,r)$, by the maximum principle we have $u_r^*>0$ in $\overline{B}(x,r)$. Therefore \eqref{Jbound} gives
\[
\int_{B(x,r)} (\langle A\nabla u,\nabla u\rangle +q_+^2)\le \int_{B(x,r)} (\langle A\nabla u_r^*,\nabla u_r^*\rangle +q_+^2) +
\kappa r^{n+\alpha},
\]
which implies that 
\[
\int_{B(x,r)} \langle A\nabla u,\nabla u\rangle \le  \int_{B(x,r)} \langle A\nabla u_r^*,\nabla u_r^*\rangle   
+
\kappa r^{n+\alpha}.
\]
Hence, since $A(x)=I$ and then by \eqref{1.25g}, 
\begin{align}
\int_{B(x,r)}|\nabla u|^2& =\int_{B(x,r)}\langle (A(x)-A(y))\nabla u,\nabla u\rangle 
+\int_{B(x,r)}\langle A(y)\nabla u,\nabla u\rangle\nonumber \\
&\le Cr^{\alpha}\int_{B(x,r)}|\nabla u|^2+\int_{B(x,r)}\langle A\nabla u_r^*,\nabla u_r^*\rangle 
+ 
\kappa r^{n+\alpha}.\label{gradbd}
\end{align}
As in \eqref{2.3mult}, 
$\int_{B(x,r)}|\nabla u_r^*|^2=\int_{B(x,r)}\langle \nabla u,\nabla u_r^*\rangle$, hence \eqref{gradbd} yields
\begin{align}
\int_{B(x,r)} &|\nabla u-\nabla u_r^*|^2 
=\int_{B(x,r)} |\nabla u|^2- \int_{B(x,r)}|\nabla u_r^*|^2
\nonumber\\
&\le Cr^{\alpha}\int_{B(x,r)}|\nabla u|^2
+\int_{B(x,r)}\langle A\nabla u_r^*,\nabla u_r^*\rangle 
+ 
\kappa r^{n+\alpha} - \int_{B(x,r)}|\nabla u_r^*|^2
\nonumber \\
&= Cr^{\alpha}\int_{B(x,r)}|\nabla u|^2
+\int_{B(x,r)}\langle (A-I)\nabla u_r^*,\nabla u_r^*\rangle
+\kappa r^{n+\alpha}
\nonumber\\
&\le Cr^{\alpha}\int_{B(x,r)}|\nabla u|^2
+ Cr^{\alpha}\int_{B(x,r)} |\nabla u_r^*|^2
+ 
\kappa r^{n+\alpha}\nonumber\\
&\le Cr^{\alpha}\int_{B(x,r)}|\nabla u|^2 
+ 
\kappa r^{n+\alpha},\label{equiv3.4}
\end{align}
by the minimizing property of $u_r^*$.

Defining $\omega(u,x,s)$ for $0<s\le r$ as in \eqref{w}, 
the triangle inequality, subharmonicity of $|\nabla u_r^*|^2$ and \eqref{equiv3.4} yield
as for \eqref{firstw}, but with a smaller error term 
\begin{align}
\omega(u,x,s) &\le \left(1+C\left(\frac{r}{s}\right)^{\frac{n}{2}}r^{\alpha/2}\right)\omega (u,x,r)+C\left(\frac{r}{s}\right)^{\frac{n}{2}}r^{\alpha/2}.\label{equiv3.5}
\end{align}

Set $r_j=2^{-j}r$ for $j\ge 0$ and apply \eqref{equiv3.5} repeatedly. 
This time the error term yields a converging series, and we obtain
as in (3.6) of \cite{davidtoroalmostminimizers},
\begin{equation}\label{equiv3.6}
\omega(u,x,r_{j+1})\le \omega(u,x,r)\prod_{l=0}^j\left(1+C2^{n/2}r_l^{\alpha/2}\right) +C2^{n/2}\sum_{l=1}^{j+1}\left(\prod_{k=l}^j\left(1+C2^{n/2}r_k^{\alpha/2}\right)\right)r_{l-1}^{\alpha/2}.
\end{equation}
Since $\prod_{l=0}^{\infty}\left(1+C2^{n/2}r_l^{\alpha/2}\right)\le C$, where $C$ depends on an upper bound for $r$, \eqref{equiv3.6} yields
\begin{equation}\label{equiv3.7}
\omega(u,x,r_{j+1})\le C\omega(u,x,r)+C2^{n/2}\sum_{l=1}^{j+1}r_{l-1}^{\alpha/2}\le C\omega(u,x,r)+Cr^{\alpha/2}.
\end{equation}
Consequently, applying this for $j$ such that 
$r_{j+1}< s \le r_j$,
\begin{equation}\label{equiv3.8}
\omega(u,x,s)\le C\omega(u,x,r)+Cr^{\alpha/2} 
\ \text{ for } 0<s \le r.
\end{equation}

Recall that all of this holds if $\overline{B}(x,2\Lambda^{1/2}\lambda^{-1/2}r)\subset \{u>0\}$ and $A(x)=I$.
Now assume that $x\in \Omega$, but maybe $A(x) \neq I$. By Lemma \ref{lem: changingvariables} and 
Remark \ref{rem: changingvariables}, $u_x$ is an almost minimizer in the domain $\Omega_x = T_x(\Omega)$, 
with the functional $J_x$ associated to $A_x$ defined by \eqref{defux}, the same exponent $\alpha$ and 
the constant $\wt\kappa = \det A(x)^{-1/2}\kappa \leq C \kappa$. 
The proof above yields
\begin{equation}\label{equiv3.8g}
\omega(u_x,x,s)\le C\omega(u_x,x,r)+Cr^{\alpha/2} %GG , 
\ \text{ for } 0<s \le r,
\end{equation}
as soon as $\overline{B}(x,2\Lambda_x^{1/2}\lambda_x^{-1/2}r)\subset \{u_x>0\}$,
where the constants $\Lambda_x$ and $\lambda_x$ are slightly different, because they correspond
to $A_x$. On the other hand $\Lambda_x$ and $\lambda_x$ are bounded above and below in terms of 
$\Lambda$ and $\lambda$.
Let us not do the precise computation but choose $c(\lambda,\Lambda)\in (0,1/2)$ such that 
this happens for $r \leq 2c(\lambda,\Lambda) \dist(x,\p\Omega)$. Note that in this case $E_x(x,r)\subset B(x, \Lambda_x^{1/2}r)\subset B(x, C_0\Lambda^{1/2}r)\subset \Omega$.
If in addition $x$ is a Lebesgue point for $|\nabla u|^2$ (recall that this happens for almost every $x\in \Omega$,
because $|\nabla u|^2 \in L^1_{loc}(\Omega)$), then 
\begin{align} \label{limnabla} %GG similar to equiv3.18
|\nabla u|^2(x) &= \lim_{s \to 0}\fint_{E_x(x,s)} |\nabla u|^2
\leq C \limsup_{s \to 0}\fint_{B(x,s)} |\nabla u_x|^2 = C \limsup_{s \to 0} \omega(u_x,x,s)^2 
\nonumber\\
&\leq C\big(\omega(u_x,x,r)+r^{\alpha/2}\big)^2
\leq C \fint_{E_x(x,r)} |\nabla u|^2  +Cr^{\alpha}
\end{align}
because $\nabla u$ and $\nabla u_x$ are related by \eqref{defux}, and by \eqref{equiv3.8g}.

Note that the Lebesgue points (with the strong definition where we average
$|u(x)-u(y)|$ on small balls) are the same for the balls and the ellipsoids $E_x(x,r)$, which have
bounded eccentricities.
Now \eqref{limnabla} means that locally, the gradient of $u$ is bounded, and hence $u$ is Lipschitz in small balls.
In particular given a compact subset $K$ of $\Omega$ there exists $\eta_K=C(\lambda,\Lambda,\dist(K,\partial\Omega)$ such that $(K, \eta_K)\subset\Omega$ and
\begin{equation}\label{tt-5}
\sup_{x\in K}|\nabla u(x)|\le C(K,\lambda,\Lambda)\int_{(K,\eta_K)}|\nabla u|^2 + C(K).
\end{equation}
Theorem \ref{T:Lippositive} follows.
\end{proof}

We shall now improve Theorem \ref{T:Lippositive} and prove that $u$ is $C^{1,\beta}$ away
from the free boundary. Before we wanted bounds on averages of $|\nabla u|^2$, and now we want
to be more precise and control the variations of $\nabla u$. Our main tool will be a (more careful)
comparison with the harmonic approximation $(u_x)_r^*$.

\medskip 
\begin{theorem}\label{T:C1beta}
Let $u$ be an almost minimizer for $J$ in $\Omega$ and set $\beta=\frac{\alpha}{n+2+\alpha}$. 
Then $u$ is of class 
$C^{1,\beta}$ locally in $\{u>0\}$ and in $\{u<0\}$.
\end{theorem}

\begin{proof}
As before we consider almost minimizers for $J$ and the open set $\{u>0\}\subset\Omega$, but the proof
works in the other cases. Let $x\in \Omega$ be given, assume first that $A(x) = I$ (we will reduce to that case later),
and let $r$ be such that $\overline B(x,r) \subset \{u>0\}\subset\Omega$.
Let $u_r^\ast$ denote, as before, the harmonic extension of $u|_{\partial B(x,r)}$ to $B(x,r)$. Then $\nabla u_r^*$ is also 
harmonic and by the mean value property
\begin{equation}\label{vuxr}
v(u,x,r):=\fint_{B(x,r)} \nabla u_r^\ast =\nabla  u_r^\ast(x).
\end{equation}
We want to estimate $\int_{B(x,\tau r)}|\nabla u-v(u,x,r)|^2$, where $\tau \in (0,1/2)$ is a small number to be chosen later, depending on $r$. 
As in 3.20 from \cite{davidtoroalmostminimizers}, we deduce that for $y\in B(x,\tau r)$,
\begin{align}
|\nabla u_r^\ast(y)-v(u,x,r)|&
=|\nabla u_r^\ast(y)-\nabla u_r^\ast(x)| \le \tau r\sup_{B(x,\tau r)}|\nabla^2 u_r^\ast|\nonumber\\
&\le C\tau \left(\fint_{B(x,r)}|\nabla u_r^\ast|\right)\le C\tau \left(\fint_{B(x,r)}|\nabla u_r^\ast|^2\right)^{1/2}\nonumber\\
&\le C\tau \left(\fint_{B(x,r)}|\nabla u|^2\right)^{1/2}= C\tau \omega(u,x,r),
\label{equiv3.20}
\end{align}
where the last part uses the Dirichlet minimality of $u_r^\ast$. Then by \eqref{equiv3.4} and \eqref{equiv3.20},
\begin{align}
\int_{B(x,\tau r)}|\nabla u-v(u,x,r)|^2 
&\le 2\int_{B(x,\tau r)}|\nabla u -\nabla (u_r^\ast)|^2+2\int_{B(x,\tau r)}|\nabla u_r^\ast-v(u,x,r)|^2
\nonumber\\
&\le 2\int_{B(x,r)}|\nabla u -\nabla (u_r^\ast)|^2 +C \tau^{n+2}r^n \omega(u,x,r)^2
\nonumber\\
& \le Cr^{\alpha}\int_{B(x,r)}|\nabla u|^2+Cr^{n+\alpha}+C \tau^{n+2}r^n \omega(u,x,r)^2.
\nonumber\\
& \le C[r^{\alpha}+  \tau^{n+2}] r^n \omega(u,x,r)^2 + Cr^{n+\alpha}
\label{equiv3.19} 
\end{align}
or, dividing by $(\tau r)^{-n}$,
\begin{equation}\label{equiv3.21} 
\fint_{B(x,\tau r)}|\nabla u-v(u,x,r)|^2 \leq C [\tau^{-n} r^\alpha + \tau^2][1+\omega(u,x,r)^2].
\end{equation}
We want to optimize in \eqref{equiv3.21} and take $\tau=r^{\frac{\alpha}{n+2}}$.
Since we required $\tau<1/2$ for the computations above, we add the assumption that 
\begin{equation}\label{rissmall}
r^{\frac{\alpha}{n+2}}<1/2
\end{equation}
Set $\rho=\tau r = r^{1+\frac{\alpha}{n+2}}=r^{\frac{n+2+\alpha}{n+2}}$, and notice that
$r^{\alpha}\tau^{-n}=\tau^2=r^{\frac{2\alpha}{n+2}}=\rho^{\frac{2\alpha}{n+2+\alpha}}$.
Also set $\beta=\frac{\alpha}{n+2+\alpha}$ as in the statement ; this way \eqref{equiv3.21} implies that
\begin{equation}\label{equiv3.21g} 
\fint_{B(x,\tau r)}|\nabla u-v(u,x,r)|^2 \leq C \rho^{2\beta} [1+\omega(u,x,r)^2].
\end{equation}
Now we want to compute everything in terms of $\rho$ rather than $r$, so we take 
\begin{equation} \label{rofrho}
r = r(\rho) = \rho^{\frac{n+2}{n+2+\alpha}}
\end{equation}
and record that \eqref{rissmall} means that $\rho < 2^{-\frac{n+2+\alpha}{\alpha}}$.
Now let
\begin{equation}\label{mux}
m(u,x,\rho)=\fint_{B(x,\rho)}\nabla u; 
\end{equation}
Since $B(x,\tau r)= B(x,\rho)$ and $m(u,x,\rho)$ 
gives the best approximation of $\nabla u$ in $L^2$, \eqref{equiv3.21g} implies that
\begin{equation}\label{equiv3.24} 
\fint_{B(x,\rho)}|\nabla u-m(u,x,\rho)|^2 \leq \fint_{B(x,\rho)}|\nabla u-v(u,x,r)|^2 
\leq C \rho^{2\beta} [1+\omega(u,x,r)^2],
\end{equation} 
where we keep $r = r(\rho)$ in $\omega(u,x,r)$ to simplify the notation.

So far this holds whenever $u(x)>0$ and $A(x)=I$, as soon as 
$\rho < 2^{-\frac{n+2+\alpha}{\alpha}}$ for \eqref{rissmall}), and 
$\overline B(x,r(\rho)) \subset \{u>0\}$, so that we can define $u_r^\ast$ and do the computations.

%We like \eqref{equiv3.24} because it says that $\nabla u$ varies less and less in small balls, and we do not fear 
%$\omega(u,x,r)$; it will be easy to estimate  because $\nabla u$ is bounded on compact subsets of  $\{u>0\}$.
%Before we do this, 
Let us extend \eqref{equiv3.24} to the case when we no longer assume that $A(x)=I$.
By Lemma \ref{lem: changingvariables} and Remark \ref{rem: changingvariables}, $u_x$ is an almost 
minimizer in the domain $\Omega_x = T_x(\Omega)$, 
with the functional $J_x$ associated to $A_x$ defined by \eqref{defux}, the same exponent $\alpha$ and 
the constant $\wt\kappa = \det A(x)^{-1/2}\kappa \leq C \kappa$. So we can apply the proof of \eqref{equiv3.24}
to the function $u_x$; we get that
\begin{equation}\label{equiv3.24bis} 
\fint_{B(x,\rho)}|\nabla u_x-m(u_x,x,\rho)|^2 \leq C \rho^{2\beta} [1+\omega(u_x,x,r(\rho))^2],
\end{equation} 
maybe with a slightly larger constant (because of $\wt\kappa$). The conditions of validity are now that
$\rho < 2^{-\frac{n+2+\alpha}{\alpha}}$, as before, and $\overline B(x,r) \subset \Omega_x$, i.e., 
\begin{equation} \label{3.20g}
\overline E_x(x,r(\rho)) = T_x^{-1}(\overline B(x,r(\rho)) \subset \Omega.
\end{equation}
Since $u=u_x \circ T$ by \eqref{defux}, $\nabla u(y) = T^t \nabla u_x(T(y))$. Then using a change of variable in \eqref{equiv3.24bis} 
we have
\begin{equation}\label{equiv3.24'} 
\fint_{E_x(x,\rho)}|\nabla u-m_E(u,x,\rho)|^2 \leq C \rho^{2\beta} [1+\omega_E(u,x,r(\rho))^2],
\end{equation}
where $C$ became larger, depending on $\lambda$ and $\Lambda$, and where
\begin{equation} \label{3,22g}
m_E(u,x,\rho) = \fint_{E_\rho(x,\rho)} \nabla u 
\, \text{ and } \, \omega_E(u,x,r)^2 = \fint_{E_\rho(x,\rho)} |\nabla u|^2.
\end{equation}

\smallskip
Now we localize and get rid of $\omega_E(u,x,r)$. 
Let $B_0 = B(x_0,r_0)$ be such that that $4B_0 \subset \{u>0\}\subset\Omega$ with
$r_0^{\frac{\alpha}{n+2}}< 1/2$, because this way we will always pick radii that satisfy 
\eqref{rissmall}. Theorem \ref{T:Lippositive} and in particular \eqref{tt-5} ensure that $u$ is Lipschitz on $2B_0$.

%
%, with
%\begin{equation} \label{lipboundg}
%||u||_{\mathrm{Lip}(2B_0)} \leq C(B_0)
%\end{equation}
%where $C$ depends on the various parameters for $J$, and no longer on $r_0$ because we put 
%an upper bound on $r_0$.

Then let $x,y \in B_0$ be given. Suppose in addition that $|x-y| \leq c r_0^{\frac{n+2+\alpha}{n+2}}$,
where the small constant $c$ depends on $\lambda$ and $\Lambda$, and will be chosen soon.
We want to apply the computations above with radii $\rho \leq 2 \lambda^{-1/2} |x-y|$, and we choose $c$ so
small that \eqref{rofrho} yields $\Lambda r(\rho) < r_0$, and so 
$E_x(x,r(\rho)) \subset B(x,r_0) \subset 2B_0$.
Then \eqref{equiv3.24'}, holds with uniform control on $\omega_E(u,x,r(\rho)) \leq C(B_0)$ by \eqref{tt-5}.

We apply this to $\rho$ and $\rho/2$, compare, and get that
\begin{align}
|m_E(u,x,\rho/2)-m_E(u,x,\rho)&|=\left|\fint_{E_x(x,\rho/2)}\nabla u-m_E(u,x,\rho)\right|
\le 2^n\fint_{E_x(x,\rho)}|\nabla u-m_E(u,x,\rho)|
\nonumber\\
&\le 2^n \left(\fint_{E_x(x,\rho)}|\nabla u-m_E(u,x,\rho)|^2\right)^{1/2}
 \le C \rho^{\beta} [1+C(B_0)^2]^{1/2}.
\label{equiv3.25}
\end{align}
Then we iterate as usual, sum a geometric series, and find that when $x$ is a Lebesgue point for $\nabla u$,
\begin{equation} \label{equiv3.26}
|\nabla u(x) - m_E(u,x,\rho)| \leq C \rho^{\beta} [1+C(B_0)^2]^{1/2}.
\end{equation}
We have a similar estimate for $y$ if $y$ is a Lebesgue point also. We now compare 
two averages as we did in \eqref{equiv3.25}. Take $\rho_x= 2 \lambda^{-1/2} |x-y|$,
so that $E_x(x,\rho_x)$ contains $B(x,2|x-y|)$, and $\rho_y = \Lambda^{-1/2} |x-y|$,
chosen so that $E_y(y,\rho_y) \subset B(y,|x-y|) \subset E_x(x,\rho_x)$. Then
\begin{align}
|m_E(u,y,\rho_y)&-m_E(u,x,\rho_x)| =\Big|\fint_{E_y(y,\rho_y)}\nabla u-m_E(u,x,\rho_x)\Big|
\nonumber\\
&
\le (\Lambda/\lambda)^{n/2} \fint_{E_x(x,\rho_x)}|\nabla u-m_E(u,x,\rho_x)|
\le C \left(\fint_{E_x(x,\rho_x)}|\nabla u-m_E(u,x,\rho_x)|^2\right)^{1/2}
\nonumber\\
&\le C \rho_x^{\beta} [1+C(B_0)^2]^{1/2} \le C |x-y|^{\beta} [1+C(B_0)^2]^{1/2}.
\label{equiv3.27}
\end{align}
This completes the proof of Theorem \ref{T:C1beta}.
\end{proof}

\section{Estimates towards Lipschitz continuity}\label{S:technical}
In this section we prove technical results needed to obtain local Lipschitz regularity for both the one phase 
and two-phase problems. 
The main case is really with two phases, but our estimates are also
true (and some times simpler) for $J^+$.

%NN: I think big modifications are needed, due to the pb below. May have ro read again.

Define the quantities \begin{equation}\label{defthebs} 
b(x,r) 
= \fint_{\partial B(x,r)}u_x\: \  \text{ and }\: \ b^+(x,r) 
=\fint_{\partial B(x,r)}|u_x|,
\end{equation}
where we recall that $u_x = u \circ T_x^{-1}$ and $T_x$ is the affine mapping from \eqref{defTE}.
We will sometimes write $b(u_x,x,r)$ and $b^+(u_x,x,r)$ to stress the dependence on $u_x$.

The object of the next manipulations will be to distinguish two types of pairs $(x,r)$, for which we will use different
estimates. 
For constants $\tau \in (0, 10^{-2}), C_0 \geq 1$, 
$C_1 \geq 3$ and 
$r_0 > 0$,
we study the class $\mathcal G(\tau, C_0, C_1 , r_0)$ of 
pairs $(x,r) \in \Omega \times (0, r_0]$ such that 
\begin{equation} \label{bballe} 
E_x(x,2r) \subset \Omega,
\end{equation}
\begin{equation}\label{defofGsetbisbig} 
 C_0 \tau^{-n} (1+ r^\alpha \omega(u_x,x,r)^2)^{1/2} \leq r^{-1} |b(x,r)|,
 \end{equation}
 and
  \begin{equation}\label{defofGsetbplusissmall} 
  b^+(x,r) \leq C_1 |b(x,r)|.
  \end{equation} 

Let us explain the idea. We force $r \leq r_0$ to have uniform estimates,
and \eqref{bballe} is natural. In \eqref{defofGsetbisbig}, we will typically choose $\tau$ very small, so 
\eqref{defofGsetbisbig} really says that the quantity $r^{-1} |b(x,r)|$ is as large as we want.
This quantity has the same dimensionality of the expected variation of $u$ on $B(x,r)$.
And in addition, \eqref{defofGsetbplusissmall} says that $b$ accounts for a significant part of $b^+$, 
which measures the average size of $|u|$. We mostly expect this to happen only far from the free boundary,
and the next lemmas go in that direction.

We will have to be a little more careful than usual, because for the first time we will play with our usual center $x$,
and at the same time with ellipsoids $E_z(z,\rho)$, with $z$ near $x$, with different orientations.
Set
\begin{equation} \label{kk}
k = \frac{1}{6}\lambda^{1/2}\Lambda^{-1/2},
\end{equation}
which we choose like this so that
\begin{equation} \label{defofkk}
E_z(z, k r) \subset B(z,\Lambda^{1/2} k r) \subset
E_x(x,r/2) \text{ whenever $x\in \Omega$ and } z \in E_x(x,r/3).
\end{equation}
Indeed recall that $E_x(x,r) = T_x^{-1}(B(x,r))$ and $T_x(y) = x + A^{-1/2}(y-x)$ by \eqref{defTE}, and similarly for $z$.
The first inclusion follows at once, and since $B(z,\Lambda^{1/2} k r)$ is contained in the translation centered at 
$z$ of $E_x(x, \lambda^{-1/2}\Lambda^{1/2} k r) = E_x(x, r/6)$, \eqref{defofkk} holds too.
We start with a self-improvement lemma.

\begin{lemma}\label{Gselfimprovement}
Assume $u$ is an almost minimizer for $J$ in $\Omega$. For each choice of constants $C_1 \geq 3$ and $r_0$, there is a constant $\tau_1 \in (0, 10^{-2})$ (which depends only on $n, \kappa, \alpha, r_0, C_1, \lambda$ and $\Lambda$), such that if $(x,r) \in \mathcal G(\tau, C_0, C_1, r_0)$ for some choice of $\tau \in (0, \tau_1)$ and $C_0 \geq 1$, 
then for each $z \in E_x(x, \tau r/3)$, 
we can find $\rho_z \in (\tau kr/2, \tau kr)$
such that $(z, \rho_z) \in \mathcal G(\tau, 10C_0, 3, r_0)$. 
Here  $k$  is defined as in \eqref{kk} and satisfies \eqref{defofkk}. 
\end{lemma}

\begin{proof} We already use $u_{x} = u \circ T_x^{-1}$ as in \eqref{defux},
and now let $(u_x)^*_r$ be the harmonic extension of $u_x $ to $B(x,r)$. Hence for $y \in B(x, \tau r)$
\begin{align} 
|(u_x)_r^*(y) - b(x,r)| &= \Big|(u_x)_r^*(y) - \fint_{\partial B(x,r)} u_x\Big| 
= |(u_x)_r^*(y) - (u_x)_r^*(x)| 
\leq \tau r \sup_{z \in B(x, \tau r)} |\nabla (u_x)_r^*(z)|
\nonumber\\ 
&\leq \tau 
\sup_{\partial B(x,r/2)} |(u_x)_r^*| \leq C\tau \fint_{\partial B(x,r)} |(u_x)_r^*| 
= C\tau \fint_{\partial B(x,r)}|u_x|\nonumber\\
& = C\tau b^+(x,r) \stackrel{\eqref{defofGsetbplusissmall}}{\leq} CC_1\tau |b(x,r)|.
\label{ustaryvsb}
\end{align}

Recall that \eqref{gradientcomparisonr} 
holds as long as $A(x)=I$ and $\overline{B}(x,r)\subset \Omega$. 
Then, by the discussion below \eqref{uaverage}, this also holds for $u_x$, as long as 
$\overline{B}(x,r)\subset \Omega_x = T_x(\Omega)$ or equivalently $\overline{E}_x(x,r)\subset \Omega$.
That is, 
\begin{equation}\label{gradientcomparisonrbis}
\int_{B(x,r)}|\nabla u_x-\nabla (u_x)_r^*|^2\le C r^{\alpha}\int_{B(x,r)}|\nabla u_x|^2+Cr^n.
\end{equation}
Then by Poincar\'{e}'s inequality and the definition \eqref{w},
\begin{equation}\label{avgtuminustsu} 
\fint_{B(x,r)} |u_x - (u_x)_r^*|^2 \leq r^2 \fint_{B(x,r)}|\nabla u_x - \nabla (u_x)_r^*|^2 
\leq Cr^2(r^\alpha \omega(u_x,x,r)^2 + 1).
\end{equation}
Applying Cauchy-Schwartz's inequality in the smaller ball; then by \eqref{avgtuminustsu} 
\begin{equation}\label{avgtuminustsuonsmall} 
\fint_{B(x,\tau r)} |u_x - (u_x)_r^*| \leq \tau^{-n/2} \left(\fint_{B(x,r)}|u_x - (u_x)_r^*|^2\right)^{1/2} \stackrel{\eqref{avgtuminustsu}}{\leq} C\tau^{-n/2}r(r^\alpha \omega(u_x,x,r)^2 +1)^{1/2},
\end{equation}
or equivalently, after an affine change of variable,
\begin{equation} \label{ams2}
\fint_{E_x(x,\tau r)} |u - (u_x)_r^* \circ T_x| \leq C\tau^{-n/2}r(r^\alpha \omega(u_x,x,r)^2 +1)^{1/2}.
\end{equation}
Now let $z\in E_x(x, \tau r/3)$ be given. We want to use \eqref{ams2} to control $b(z,\rho)$ for some $\rho \in (\tau k r/2,\tau kr)$
Fix $x$ and $z$, and notice that for each such $\rho$,
\begin{equation} \label{bjac}
b(z,\rho) = \fint_{\zeta \in \partial B(z,\rho)}u_z(\zeta)\, d\zeta = \fint_{\partial E_z(z,\rho)} u(\xi) \mathcal{J}(\xi) d\sigma(\xi), 
\end{equation}
where we set $\xi = T_z(\zeta) \in \partial E_z(z,\rho)$, notice that $u_z(\zeta) = u(\xi)$. Here $\mathcal{J}(\xi)$ is the Jacobian of the change of variable $T_z$. Since $A_x(x)=I$ and $A$ is H\"older continuous one can show that $|\mathcal{J}(\xi)-1| \leq C (\tau r)^{\alpha}$. Here we only use the fact that $C^{-1}\le\mathcal{J}(\xi)\le C$ for some $C$ depending on $\lambda$, $\Lambda$ and the H\"older norm of $A$.
% and find out
%with surprise that there is a Jacobian, $J(\xi)$, which depends on $T_z$ and on the direction of $\xi-z$,
%but fortunately is such that $C^{-1} \leq J(\xi) \leq C$; we could even show that $|J(\xi)-1| \leq C (\tau r)^{\alpha}$ 
%because $A$ is H\"older continuous, but we shall try to avoid this. 
There is no problem with the definition and the domains,
because $E_z(z,k\tau r) \subset E_x(x,\tau r/2)$ by \eqref{defofkk}.

Now we subtract $b(x,r)$, take absolute values, and integrate on $I = (\tau k r/2,\tau kr)$. We get that
\begin{align} \label{bjad}
\int_I |b(z,\rho) - b(x,r)|\, d\rho &\leq C (\tau r)^{n-1} \int_{\rho \in I} \int_{\partial E_z(z,\rho)} |u(\xi)-b(x,r)| \mathcal{J}(\xi) d\sigma(\xi)d\rho
\nonumber\\& \leq C (\tau r)^{n-1} \int_{E_x(x,\tau r/2)} |u-b(x,r)|
\leq C \tau r \fint_{E_x(x,\tau r/2)} |u-b(x,r)|.
\end{align}
Observe that  for $\xi\in E_x(x,\tau r/2)= T_x^{-1}(B(x, \tau r/2))$
\begin{equation}\label{tt-6}
|u(\xi)-b(x,r)| \leq |u(\xi) - (u_x)_r^*(T_x(\xi)| +  |(u_x)_r^*(T_x(\xi)) - b(x,r)|.
\end{equation}
Use \eqref{tt-6}, \eqref{ustaryvsb}
and \eqref{ams2}; this yields
\begin{equation} \label{bjae}
\int_I |b(z,\rho) - b(x,r)| \leq C \tau r \big[C_1 \tau |b(x,r)| + \tau^{-n/2}r(r^\alpha \omega(u_x,x,r)^2 +1)^{1/2}\big]
\end{equation}
and allows us to choose, by Chebyshev, a radius $\rho = \rho_z\in (\tau k r/2,\tau kr)$ such that
\begin{eqnarray} \label{batrhozbig}
|b(z,\rho_z) - b(x,r)| &\leq &C \big[C_1 \tau |b(x,r)| + \tau^{-n/2}r(r^\alpha \omega(u_x,x,r)^2 +1)^{1/2}\big]\nonumber\\
&\leq & C \big[C_1 \tau |b(x,r)| + \tau^{n/2} |b(x,r)| \big]\nonumber\\
&\leq & \frac{1}{2} |b(x,r)|, 
\end{eqnarray}
provided $\tau_1$ is small enough depending on $C$ (which depends on the usual constants 
for $J$) and $C_1$. Note that we have used \eqref{defofGsetbisbig} for 
 $\tau \leq \tau_1$. Note that \eqref{batrhozbig}
implies that 
\begin{equation}
|b(z,\rho_z)| \geq \frac{1}{2} |b(x,r).
\end{equation}

%\begin{equation} \label{batrhozbig}
%|b(z,\rho_z) - b(x,r)| \leq \frac12 |b(x,r)|.
%\end{equation}

Next we need to prove \eqref{defofGsetbplusissmall} for $\rho_z$, and for this we want to control 
\begin{equation} \label{bjaz}
b^+(z,\rho) = \fint_{\zeta \in \partial B(z,\rho)} |u_z(\zeta)| = \fint_{\partial E_z(z,\rho)} |u(\xi)| \mathcal{J}(\xi) d\sigma(\xi), 
\end{equation}
where the only difference with \eqref{bjac} is that we used $|u|$. 
Recall that $||a|-|b||\le |a-b|$ and continue the computation as above, putting absolute values in \eqref{ustaryvsb}
to control the integral of $|(u_x)_r^*(y)| - |b(x,r)|$ and in \eqref{ams2} to control the integral of $|u|-|(u_x)^*\circ T_x|$. We obtain as in \eqref{batrhozbig}
that
\begin{equation} \label{bjaff}
|b^+(z,\rho_z) - |b(x,r)|| \leq C \big[C_1 \tau |b(x,r)| + \tau^{-n/2}r(r^\alpha \omega(u_x,x,r)^2 +1)^{1/2}\big]
\leq \frac12 |b(x,r)|.
\end{equation}
We do not even have to ask for an additional Chebyshev requirement for $\rho_z$, even though we could have done so.
Hence, if $\tau_1$ is small enough and by \eqref{batrhozbig},
\begin{align}
b^+(z, \rho_z) \leq \frac{3}{2}|b(x,r)| \leq 3 |b(z,\rho_z)|,
\label{upperboundonbplussmallscale}
\end{align}
which is \eqref{defofGsetbplusissmall} with $C_1 = 3$. 

We are left to verify 
an analogue of \eqref{defofGsetbisbig} at the scale $\rho_z$, and for this 
we control $\omega(u_z,z, \rho_z)$ in terms of $\omega(u_x,x,r)$.
%To some extent, if we were only interested by local estimates, we could say that 
%$\omega(u_z,z, \rho_z) \leq C$ locally, and get some estimate. 
%But anyway this will be easy. 
First observe that $E_z(z,k r) \subset E_x(x, r/2)$
by \eqref{defofkk}; hence we can apply \eqref{logboundw} to $u_z$ in 
$\overline B(z,k r)$, between the radii $\rho_z$ and $k r$; we get that
\[
\omega(u_z,z,\rho_z)\le C\omega(u_z,z, kr)+C\log(k r/\rho_z) \leq C\omega(u_z,z, kr)+C (1+|\log(\tau)|).
\]
Then by \eqref{w} 
\[
\omega(u_z,z, kr)^2 = \fint_{B(z,kr)} |\nabla u_z|^2 \leq C \fint_{E_z(z,kr)} |\nabla u|^2
\leq C \fint_{E_x(x,r)} |\nabla u|^2 \leq C \omega(u_x,x,r)^2
\]
with constants $C$ that depend also on $\lambda$ and $\Lambda$, so 
\begin{equation} \label{boundedtwzbytwx}
\omega(u_z,z,\rho_z) \leq C \omega(u_x,x,r) + C (1+|\log(\tau)|).
\end{equation}

Thus
\begin{align} 
1+ \rho_z^\alpha \omega(u_z,z, \rho_z)^2 &  \stackrel{\eqref{boundedtwzbytwx}}{\leq} 
1+ C\rho_z^\alpha \omega(u_x,x,r)^2 + C\rho_z^{\alpha}(1+|\log\tau|)^2\nonumber\\
&\leq 1 + Cr^\alpha\omega(u_x,x,r)^2 + C(\tau r)^\alpha (1+|\log \tau|)^2\nonumber\\
&\leq 1 + Cr^\alpha \omega(u_x,x,r)^2 + Cr^\alpha [\tau^\alpha (1+|\log \tau|)^2] \nonumber\\
&\leq C(1+r^\alpha \omega(u_x,x,r)^2) \stackrel{\eqref{defofGsetbisbig}}{\leq} \left(\frac{C\tau^n}{C_0r} |b(x,r)|\right)^2 \stackrel{\eqref{batrhozbig}}{\leq} \left(\frac{2C\tau^n}{C_0r}|b(z, \rho_z)|\right)^2.\label{finalboundsontwatrho}\end{align}

Recall that $r \simeq \tau^{-1}\rho_z$ with constants of comparability depending only on $n, \lambda, \Lambda$. Together with \eqref{finalboundsontwatrho}, this remark yields
\begin{align} 
 |b(z, \rho_z)| &\geq \frac{C_0r}{2C\tau^{n}}\left(1+\rho_z^\alpha \omega(u_z,z, \rho_z)^2\right)^{1/2}\nonumber\\
&\geq \frac{C_0\rho_z}{2C\tau^{n+1}}\left(1+\rho_z^\alpha \omega(u_z,z, \rho_z)^2\right)^{1/2}\nonumber\\
&= (C\tau)^{-1}C_0\tau^{-n} \rho_z\left(1+\rho_z^\alpha \omega(u_z,z, \rho_z)^2\right)^{1/2}.\label{bbigatlowerscale}
\end{align}
Here 
$C > 0$ is a constant which depends on $n, \lambda$ and $\Lambda$. 
Therefore we can choose $\tau_1$ so small that  
$(C\tau)^{-1} \geq 10$ above. 
Thus we have \eqref{defofGsetbisbig} at $(z,\rho_z)$ with the constant $10C_0$. We can conclude that $(z,\rho_z) \in \mathcal G(\tau, 10C_0, 3, r_0)$, which is the desired result. 
\end{proof} 

\begin{lemma}\label{samesignatsmallscale} Let $u, x,r$ satisfy the hypothesis of 
Lemma \ref{Gselfimprovement}; in particular $(x,r) \in \mathcal G(\tau, C_0, C_1, r_0)$ for some $C_0 \geq 1$, $C_1\ge 3$ and $\tau \leq \tau_1$. Recall that $b(x,r) \neq 0$ by \eqref{defofGsetbisbig}. 
If $b(x,r) > 0$ then 
\begin{align}
u\geq 0 \mbox{ on } E_x(x,\tau r/3) 
%NN I think E_x(x,\tau kr)  is safer but I try this. same number as in the previous lemma
%NN was : E_x(x,\tau r \lambda^{1/2}\Lambda^{-1/2}/2) 
\mbox{ and } u > 0 \mbox{ almost everywhere on } E_x(x,\tau r/3) % same thing
%E_x(x,\tau r \lambda^{1/2}\Lambda^{-1/2}/2).
\label{eqn:postitiveatsmallscale} 
\end{align} Similarly, if $b(x,r)< 0$, then
\begin{equation}\label{eqn:negativeatsmallscale} 
u\leq 0 \mbox{ on } E_x(x,\tau r/3) % idem
%E_x(x,\tau r \lambda^{1/2}\Lambda^{-1/2}/2) 
\mbox{ and } u  < 0 \mbox{ almost everywhere on } E_x(x,\tau r/3).
%E_x(x,\tau r \lambda^{1/2}\Lambda^{-1/2}/2).
\end{equation}
\end{lemma}

\begin{proof}
Let $z\in E_x(x,\tau r/3)$. %Idem %E_x(x,\tau r \lambda^{1/2}\Lambda^{-1/2}/2)$. 
Apply Lemma \ref{Gselfimprovement} to get $(z, \rho_z) \in \mathcal G(\tau, 10C_0, 3, r_0)$. Let $\rho_z = \rho_0$. Iterate Lemma \ref{Gselfimprovement}, $j$ times, each time around the point $z$, to get $(z, \rho_j) \in \mathcal G(\tau, 10^jC_0, 3, r_0)$ where $\rho_j \in ( (\tau k/2)^j r, (\tau k)^j r)$.
%NN removed \tilde in $\rho_j \in ( (\tau\tilde{k}/2)^j r, (\tau\tilde{k})^j r)$. 
By \eqref{defofGsetbisbig}, %NN no new paragraph and a comma 
\begin{equation}\label{bbigatrhoj} 
\rho_j^{-1}|b(z, \rho_j)| \geq 10^j C_0 \tau^{-n}(1+\rho_j^\alpha w(u_z,z, \rho_j)^2)^{1/2}.
\end{equation} 

Arguing as before (i.e. obtaining \eqref{batrhozbig} at the
%NN was : \eqref{brhominusbstarrho} and \eqref{bstarminusbatr} at 
scale $j$) we see that 
\[
|b(z, \rho_j) - b(z, \rho_{j-1})| < \frac{1}{2}|b(z, \rho_{j-1})|,
\]
 that is, $b(z,\rho_j)$ has the same sign as $b(z, \rho_{j-1})$. An induction argument yields that $b(z,\rho_j)$ has the same sign as $b(x,r)$ for all $j$. Set 
\begin{equation}\label{defofZj} 
Z_j = \{y \in B(z, \tau \rho_j)\mid u(y)b(x,r) \leq 0\}
= \{y\in B(z, \tau \rho_j)\mid u(y)b(z,\rho_j)\le 0\}.
\end{equation} 
One should think of this as the subset of $B(z,\tau \rho_j)$ where $u$ has the ``wrong" sign. In particular if $y\in Z_j$, $|u(y)-b(z,\rho_j)|\ge |b(z,\rho_j)|$.
Arguing exactly as in the proof of Lemma \ref{Gselfimprovement} we can prove as in \eqref{ustaryvsb}
(and because we took $\tau$ small enough for \eqref{batrhozbig}) that 
\begin{equation}\label{ustarminusbatj}
|(u_z)^*_{\rho_j}(y) - b(z, \rho_j)| \leq CC_1 \tau |b(z, \rho_j)| 
\leq \frac{1}{4} |b(z, \rho_j)| 
\;\;\; \text{ for }
y\in B(z, \tau \rho_j).
\end{equation} 
Here $(u_z)^*_{\rho_j}$ is the harmonic extension of $u_z$ to $B(z,\rho_j)$.
This implies that $(u_z)^*_{\rho_j}$ shares a sign with $b(z, \rho_j)$ on $B(z,\tau \rho_j)$. Thus, for every $y\in Z_j$ we have 
\begin{equation}\label{uminusustaronZ} |u_z(y) - (u_z)^*_{\rho_j}(y)| \geq |u_z(y)- b(z,\rho_j)| - |b(z, \rho_j)-(u_z)^*_{\rho_j}(y)| \geq \frac{3}{4}|b(z, \rho_j)|.\end{equation} 
In other words, 
$$Z_j \subset \Big\{y\in B(z, \tau \rho_j) \ : \  |u_z(y) - (u_z)^*_{\rho_j}(y)| \geq \frac{3}{4}|b(z, \rho_j)|\Big\}.$$ 
Arguing as in \eqref{ams2}, Markov's inequality combined with \eqref{avgtuminustsuonsmall} yields,  
\begin{eqnarray}\label{estimatethesizeofZtwo}
|Z_j| &\leq &\frac{4}{3|b(z,\rho_j)|} \int_{B(z,\tau \rho_j)} |u_z - (u_z)^*_{\rho_j}|\nonumber\\
&\leq & \frac{C}{|b(z,\rho_j)|}\int_{E_z(z,\tau \rho_j)} |u - (u_z)^*_{\rho_j}\circ T_z| \nonumber\\
&\leq &  \frac{C}{|b(z,\rho_j)|}(\tau \rho_j)^n\tau^{-n/2} \rho_j(\rho_j^\alpha w(u_z,z,\rho_j)^2 + 1)^{1/2}\nonumber\\
&\leq &C(\tau \rho_j)^n 
\frac{
\tau^{-n/2}\rho_j(1+\rho_j^\alpha w(u_z,z, \rho_j)^2)^{1/2}}{|b(z,\rho_j)|}\nonumber\\ 
&\leq&C[C_010^{j}]^{-1}(\tau \rho_j)^n \tau^{n/2}\frac{10^j C_0 \tau^{-n}(1+\rho_j^\alpha w(u_z,z,\rho_j)^2)^{1/2}}{\rho_j^{-1}|b(z, \rho_j)|}\nonumber\\
&\stackrel{\eqref{bbigatrhoj}}{\leq}&C[C_010^{j}]^{-1}(\tau \rho_j)^n \tau^{n/2}.
\end{eqnarray}

%Arguing as in \eqref{ams2}, we have
%$$\int_{E_z(z,\tau \rho_j)} |u - (u_z)^*_{\rho_j}\circ T_z| \leq C(\tau \rho_j)^n\tau^{-n/2} \rho_j(\rho_j^\alpha w(u_z,z,\rho_j)^2 + 1)^{1/2},$$ 
%which implies, by Markov's inequality, 
%\begin{align}
%|Z_j| &\leq C(\tau \rho_j)^n 
%\frac{
%\tau^{-n/2}\rho_j(1+\rho_j^\alpha w(u_z,z, \rho_j)^2)^{1/2}}{|b(z,\rho_j)|}\nonumber\\ 
%&= C[C_010^{j}]^{-1}(\tau \rho_j)^n \tau^{n/2}\frac{10^j C_0 \tau^{-n}(1+\rho_j^\alpha w(u_z,z,\rho_j)^2)^{1/2}}{\rho_j^{-1}|b(z, \rho_j)|}\nonumber\\
%&\stackrel{\eqref{bbigatrhoj}}{\leq}C[C_010^{j}]^{-1}(\tau \rho_j)^n \tau^{n/2}.
%\label{estimatethesizeofZtwo} 
%\end{align}

To simplify the discussion assume that $b(x,r) > 0$ and thus $b(z, \rho_j) > 0$ for all $j$. 
Then $Z_j = \{u \leq 0\} \cap B(z,\tau \rho_j)$. Divide both sides of \eqref{estimatethesizeofZtwo} by $|B(z, \tau \rho_j)| \simeq (\tau \rho_j)^n$ and then let $j \rightarrow \infty$ to get that 

\begin{equation}\label{nolesbesguepoints}
\lim_{j\rightarrow \infty} \frac{|\{u \leq 0\} \cap B(z, \tau \rho_j)|}{|B(z,\tau \rho_j)|} = \lim\limits_{j\rightarrow\infty}\frac{C|Z_j|}{(\tau \rho_j)^n} = 0 \;\;\; \text{ for all } z\in E_x(x,\tau r/3).
\end{equation}
Thus $|\{u\le 0\}\cap E_x(x,\tau r/3)|=0$. Hence $u>0$ a.e on $E_x(x,\tau r/3)$ and by continuity $u\ge 0$ on $E_x(x,\tau r/3)$.
%By the Lesbesgue differentiation theorem we conclude that $u > 0$ almost everywhere in $E_x(x,\tau r \lambda^{1/2}\Lambda^{-1/2}/2)$. We should note that we can differentiate by the above ellipses (instead of by balls) because they are a family with bounded eccentricity.  By the continuity of $u$ we have that $u \geq 0$ on 
%$E_x(x,\tau r/3)$,
%which is the desired result. 
The case where $b(x,r) < 0$ follows in the same way. 
\end{proof}

For the next lemma we use Lemma \ref{samesignatsmallscale} to get some regularity
for $u$ near a point $x$ such that $(x,r) \in \mathcal{G}(\tau,C_0,C_1,r_0)$, with the same method as
for the local regularity of $u$ away from the free boundary.

\begin{lemma}\label{L:4.3} There exist constants $ k_1\in(0, k/2)$, depending only
on $\lambda$ and $\Lambda$, and $\tau_2\in (0,\tau_1)$, with $\tau_1$ as in Lemmas \ref{Gselfimprovement} and \ref{samesignatsmallscale}
with the following properties.
Let $u$ be an almost minimizer for $J$ in $\Omega$.
%, and let $x, r$ satisfy the assumptions of 
%Lemma \ref{Gselfimprovement}, except that we may need to make $\tau_1$ smaller for this lemma. 
%In particular, 
Let $(x,r)\in \mathcal{G}(\tau,C_0,C_1,r_0)$ for some $\tau\in(0,\tau_2)$ and $C_0\ge 1$. 
Then for $z\in B(x,\tau r/10)$ and $s\in (0,k_1\tau r)$,
\begin{equation}\label{4.36}
\omega(u,z,s)\le C\left(\tau^{-\frac{n}{2}}\omega(u_x,x,r)+r^{\frac{\alpha}{2}}\right),
\end{equation}
and for $y,z\in B(x,\tau r/10)$,
\begin{equation}\label{4.37}
|u(y)-u(z)|\le C\left(\tau^{-\frac{n}{2}}\omega(u_x,x,r)+r^{\frac{\alpha}{2}}\right)|y-z|.
\end{equation}
Here $C=C(n,\kappa,\alpha,\lambda,\Lambda,r_0)$. Finally, there is a constant $C(\tau, r)$ 
depending on $n,\kappa,\alpha,r_0,\tau, r,\lambda,\Lambda$, such that
\begin{equation}\label{4.38}
|\nabla u(y)-\nabla u(z)|\le C(\tau, r)(\omega(u_x,x,r)+1)|y-z|^{\beta},
\end{equation}
 for any $y,z\in B(x,\tau r/10)$, % was twice smaller than the others
 where as before $\beta=\frac{\alpha}{n+2+\alpha}$.
\end{lemma}

\begin{proof}
 Let $u$, $x$ and $r$ be as in the statement and $z\in B(x,\tau r/3)$. 
 %%NN I try like this and will add conditions as it goes
 % $z\in B(x,2k_2\tau r)$. % works with \tau r/3 so far
 % $z\in B(x,\frac{\tau r \tilde{k}B}{2})$. 
 By taking $\tau_2$ small enough Lemmas \ref{Gselfimprovement} and \ref{samesignatsmallscale} hold when we 
 replace $k$ by $2k_1$ (see \eqref{bjad} and \eqref{bjae}).
 %At the price of making $\tau_1$ smaller in the two lemmas above, the proof of
% these lemmas is also valid when we replace $k$ with $2k_1$ (we will choose $k_1 < k/2$); 
 Thus we can find
$\rho\in (k_1 \tau r, 2 k_1 \tau r)$ such that such that $(z,\rho)\in \mathcal{G}(\tau, 10C_0,3,r_0)$.
 Since $b(x,r)\ne 0$ by \eqref{defofGsetbisbig}, we can assume $b(x,r)>0$ (the other case is similar). 
 By Lemma~\ref{samesignatsmallscale}, $u\ge 0$ in $E_x(x,\tau r /3)$ and 
 $u>0$ almost everywhere in $E_x(x,\tau r /3)$. 
 
 Assume now that $z\in B(x,\tau r/6)$, and apply \eqref{defofkk} with the radius $\tau r/2$;
 we get that $E_z(z, k \tau r/2) \subset E_x(x,\tau r /4)$ and so $u>0$ almost everywhere 
 on $E_z(z, k \tau r/2)$.
 This means that in the definition \eqref{eqn:jonballg} of our functional,
\begin{equation}\label{aaa} 
J_{E,z,k \tau r/2}(u)=
\int_{E_z(z,\tau r/2)}\left\langle A\nabla u,\nabla u\right\rangle+q^2_+ 
\end{equation} 
with a full contribution for $q^2_+$. The same thing holds for other ellipsoids contained in $E_x(x,\tau r /4)$,
and in particular smaller ellipsoids centered at $z$.
%%NN pay more attention here, since I destroy most of the prof below?
In Section \ref{S:C1beta}, positivity almost everywhere and its consequence \eqref{aaa}
were the only way we ever used the fact that ellipsoids are contained in $\{u_x>0\}$. That is,
we can repeat the proofs of that section as long as our ellipsoids stay inside $E_x(x,\tau r /4)$. 
In particular, if we choose $k_2$ small enough (depending on $\lambda$ and $\Lambda$), 
and set $r_2 = k_2 \tau r$, the proof of \eqref{equiv3.8g} also yields
\begin{equation} \label{equiv3.8gb}
\omega(u_z,z,t)\le C\omega(u_z,z,r_2)+Cr_2^{\alpha/2} 
\ \text{ for } 0<t \le r_2,
\end{equation}
because our earlier condition that $\overline{B}(z,2\Lambda_z^{1/2}\lambda_z^{-1/2}r_2)
\subset E_z(z,\tau r/2)$, where $\lambda_z$ and $\Lambda_z$ are easily estimated in terms of
$\lambda$ and $\Lambda$, is satisfied. Now observe that
$\omega(u,z,s) \leq C \omega(u_z,z, \lambda^{-1/2} s)$, by \eqref{w} and \eqref{defTE},
and $\omega(u_z,z,r_2) \leq C \tau^{-\frac{n}{2}}\omega(u_x,x,r)$, for the same reasons;
\eqref{4.36} follows provided $   s\le \lambda^{1/2}k_2\tau r$. Thus $k_1=k_2\lambda^{1/2}$.

Next \eqref{4.37} follows from \eqref{4.36}, because $\nabla u(z)$ can be computed almost 
everywhere as limits of averages of $u$, which are dominated by $\limsup_{s \to 0} \omega(u,z,s)$.

The local H\"older estimate for $\nabla u$ in \eqref{4.38} is very similar to the proof of Theorem \ref{T:C1beta}.
We only need to make sure that we never get outside of the ellipsoid $E_x(x,\tau r /4)$,
where we know that $u > 0$ almost everywhere.
\end{proof}

\begin{lemma}\label{L:4.4} Let $u$ be an almost minimizer for $J$ in $\Omega$. There exists $K_2=K_2(\lambda,\Lambda)\ge 2$ such that for each choice of $\gamma\in(0,1)$, $\tau>0$ and $C_0\ge 1$, we can find $r_0,\eta$ small and $K\ge 1$ with the following property: if $x\in\Omega$ and $r>0$ are such that $0<r\le r_0$, $B(x,K_2r)\subset \Omega$ and
\begin{equation}\label{4.55}
|b(u_x,x,r)|\ge \gamma r(1+\omega(u_x,x,r)),
\end{equation}
and
\begin{equation}\label{4.56}
\omega(u_x,x,r)\ge K,
\end{equation}
then there exists $\rho\in \left(\frac{\eta r}{2},\eta r\right)$ such that $(x,\rho)\in \mathcal{G}(\tau, C_0, 3,r_0)$.
\end{lemma}

\begin{proof} Let $\eta\in (0,10^{-2})$ be small, to be chosen later, and let $(x,r)$ be as in the statement. Let $(u_x)_r^*$ be the harmonic extension of $u_x$ to $ B(x,r)$. Notice that $|\nabla (u_x)_r^*|^2$ is subharmonic on $B(x,r)$, and $\int_{B(x,r)}|\nabla (u_x)_r^*|^2\le \int_{B(x,r)}|\nabla u_x|^2$. For $y\in B(x,\eta r)$,
\begin{equation}\label{4.57}
|\nabla (u_x)_r^*(y)|^2\le\fint_{B(y,r/2)}|\nabla (u_x)_r^*|^2\le 2^n\fint_{B(x,r)}|\nabla (u_x)_r^*|^2\le 2^n\fint_{B(x,r)}|\nabla u_x|^2= 2^n\omega(u_x,x,r)^2.
\end{equation}
Since $(u_x)_r^*$ is harmonic in $B(x,r)$, $\displaystyle{(u_x)_r^*(x)=\fint_{\partial B(x,r)}u_x=b(u_x,x,r)}$. Therefore for $y\in B(x,\eta r)$.
\begin{equation}\label{4.58}
|(u_x)_r^*(y)-b(u_x,x,r)|=|(u_x)_r^*(y)-(u_x)_r^*(x)|\le \eta r\sup_{B(x,\eta r)}|\nabla (u_x)_r^*|\le 2^{n/2}\eta r\omega(u_x,x,r).
\end{equation}
We will choose $\eta$ so small that $2^{n/2}\eta<\gamma/4$. Then \eqref{4.55} and \eqref{4.58} yield
\begin{equation}\label{5.59}
|(u_x)_r^*(y)-b(u_x,x,r)|\le 2^{n/2}\eta r\omega(u_x,x,r)\le\frac{1}{4}\gamma r\omega(u_x,x,r)\le \frac{1}{4}|b(u_x,x,r)|.
\end{equation}
In particular, $(u_x)_r^*$ has the same sign as $b(u_x,x,r)$ on $B(x,\eta r)$ and
\begin{equation}\label{4.60}
\frac{5}{4}|b(u_x,x,r)|\ge|(u_x)_r^*(y)|\ge\frac{3}{4}|b(u_x,x,r)| \ \text{ for } \ y\in B(x,\eta r).
\end{equation}

Since 
\[
\int_{B(x,\eta r)\setminus B(x,\eta r/2)}|u_x-(u_x)_r^*|=\int_{\eta r/2}^{\eta r}\int_{\partial B(x,s)}|u_x-(u_x)_r^*|,
\]
there exists $\rho\in \left(\frac{\eta r}{2},\eta r\right)$ such that
\[
\int_{\partial B(x,\rho)}|u_x-(u_x)_r^*|\le \frac{2}{\eta r}\int_{B(x,\eta r)\setminus B(x,\eta r/2)}|u_x-(u_x)_r^*|=\frac{2}{\eta r}\int_{\eta r/2}^{\eta r}\int_{\partial B(x,s)}|u_x-(u_x)_r^*|.
\]
Poincar\'{e}'s inequality and Cauchy-Schwarz lead to
\begin{align}
\int_{\partial B(x,\rho)}|u_x-(u_x)_r^*|&\le \frac{2}{\eta r}\int_{\eta r/2}^{\eta r}\int_{\partial B(x,s)}|u_x-(u_x)_r^*|\le \frac{2}{\eta r}\int_{B(x,\eta r)}|u_x-(u_x)_r^*|\nonumber\\
&\le C\int_{B(x,\eta r)}|\nabla u_x-\nabla (u_x)_r^*|\le C(\eta r)^{n/2}\left(\int_{B(x,\eta r)}|\nabla u_x-\nabla (u_x)_r^*|^2\right)^{1/2}\nonumber\\
&\le C(\eta r)^{n/2}\left(\int_{B(x,r)}|\nabla u_x-\nabla (u_x)_r^*|^2\right)^{1/2}.\label{4.61}
\end{align}

By \eqref{gradientcomparison3}
\begin{equation}\label{4.62}
\fint_{B(x,r)}|\nabla u_x-\nabla (u_x)_r^*|^2\le Cr^{\alpha}\fint_{B(x,r)}|\nabla u_x|^2+C=Cr^{\alpha}\omega(u_x,x,r)^2+C.
\end{equation}
Combining \eqref{4.62} and \eqref{4.61} yields
\begin{equation}\label{4.63}
\int_{\partial B(x,\rho)}|u_x-(u_x)_r^*|\le C\eta^{n/2}r^n(1+r^{\alpha}\omega(u_x,x,r))^{1/2}.
\end{equation}
Since $r\le r_0$, then $r^{\alpha}\le r_0^{\alpha}$ and by \eqref{4.63} and \eqref{4.56},
\begin{align}
\fint_{\partial B(x,\rho)}|u_x-(u_x)_r^*|&\le C(\eta r)^{1-n}\int_{\partial B(x,\rho)}|u_x-(u_x)_r^*|\le C\eta^{1-\frac{n}{2}}r(1+r_0^{\alpha}\omega(u_x,x,r)^2)^{1/2}\nonumber\\
&\le C\eta^{1-\frac{n}{2}}r\omega(u_x,x,r)(K^{-2}+r_0^{\alpha})^{1/2}.\label{4.64}
\end{align}
We choose $K$ large enough and $r_0$ small enough, both depending on $\gamma$ and $\eta$ (recall that $\eta$ depends on $\gamma$ only), so that in \eqref{4.64},
\begin{equation}\label{4.65}
C\eta^{1-\frac{n}{2}}(K^{-2}+r_0^{\alpha})^{1/2}\le\frac{\gamma}{4}.
\end{equation}
Then by \eqref{4.64} and \eqref{4.55},
\begin{equation}\label{4.66}
\fint_{\partial B(x,\rho)}|u_x-(u_x)_r^*|\le\frac{\gamma}{4}r\omega(u_x,x,r)\le\frac{|b(u_x,x,r)|}{4}.
\end{equation}
As mentioned above, $(u_x)_r^*$ has the same sign as $b(u_x,x,r)$ in $B(x,\eta r)$. Since $\rho<\eta r$, $(u_x)_r^*$ does not change sign in $\partial B(x,\rho)$. By \eqref{4.60} and \eqref{4.66},
\begin{align}
|b(u_x,x,\rho)|&=\Big|\fint_{\partial B(x,\rho)}u_x\Big|\ge \Big|\fint_{\partial B(x,\rho)}(u_x)_r^*\Big|-\fint_{\partial B(x,\rho)}|u_x-(u_x)_r^*|\nonumber\\
&=\fint_{\partial B(x,\rho)}|(u_x)_r^*|-\fint_{\partial B(x,\rho)}|u_x-(u_x)_r^*|\nonumber\\
&\ge \frac{3}{4}|b(u_x,x,r)|-\frac{1}{4}|b(u_x,x,r)|=\frac{1}{2}|b(u_x,x,r)|.\label{4.67}
\end{align}
The same computations yield
\begin{align}
|b^+(u_x,x,\rho)|&=\fint_{\partial B(x,\rho)}|u_x|\le\fint_{\partial B(x,\rho)}|(u_x)_r^*|+\fint_{\partial B(x,\rho)}|u_x-(u_x)_r^*|\nonumber\\
&\le\frac{5}{4}|b(u_x,x,r)|+\frac{1}{4}|b(u_x,x,r)|\le\frac{3}{2}|b(u_x,x,r)|.\label{4.68}
\end{align}
This shows that $(x,\rho)$ satisfies \eqref{defofGsetbplusissmall}  with $C_1=3$. We still need to check \eqref{defofGsetbisbig}.
By \eqref{4.67} and \eqref{4.55},

\begin{equation}\label{4.69}
\frac{|b(u_x,x,\rho)|}{\rho}\ge\frac{1}{2\rho}|b(u_x,x,r)|\ge\frac{\gamma r}{2\rho}(1+\omega(u_x,x,r))\ge\frac{\gamma}{2\eta}(1+\omega(u_x,x,r)).
\end{equation}
We now need a lower bound for $\omega(u_x,x,r)$ in terms of $\omega(u_x,x,\rho)$.
Applying \eqref{iterationw} to $u_x$ (which can be done as long as $B(x,r)\subset\Omega_x$), for any $j\ge 0$ integer, we have
\begin{equation}\label{4.70}
\omega(u_x,x,2^{-j-1}r)\le C\omega(u_x,x,r)+Cj
\end{equation}
We apply this to the integer $j$ such that $2^{-j-2}r\le \rho<2^{-j-1}r$ and get, recalling that $\rho\in(\frac{\eta r}{2}, \eta r)$, that
\begin{equation}\label{4.71}
\omega(u_x,x,\rho)\le 2^{n/2}\omega(u_x,x,2^{-j-1}r)\le C(\omega(u_x,x,r)+Cj)\le \omega(u_x,x,r)+C|\log\eta|.
\end{equation}
Then \eqref{4.69} yields
\begin{align}
(1+\rho^{\alpha}\omega(u_x,x,\rho)^2)^{1/2}&\le 1+\rho^{\alpha/2}\omega(u_x,x,\rho)\le 1+Cr_0^{\alpha/2}\omega(u_x,x,r)+Cr_0^{\alpha/2}|\log\eta|\nonumber\\
&\le (1+Cr_0^{\alpha/2}+Cr_0^{\alpha/2}|\log\eta|)(1+\omega(u_x,x,r))\nonumber\\
&\le (1+Cr_0^{\alpha/2}+Cr_0^{\alpha/2}|\log\eta|)\frac{2\eta}{\gamma}\frac{|b(u_x,x,\rho)|}{\rho}.\label{4.72}
\end{align}
Multiplying by $C_0\tau^{-n}$ we obtain
\begin{equation}\label{4.73}
C_0\tau^{-n}(1+\rho^{\alpha}\omega(u_x,x,\rho)^2)^{1/2}\le C_2\frac{|b(u_x,x,\rho)|}{\rho},
\end{equation}
where 
\begin{equation}\label{4.74}
C_2=C_0\tau^{-n}(1+Cr_0^{\alpha/2}+Cr_0^{\alpha/2}|\log\eta|)\frac{2\eta}{\gamma}.
\end{equation}
This shows that \eqref{defofGsetbisbig} holds for $(x,\rho)$ if $C_2\le 1$. We choose $\eta$ so small (depending on $C_0, \tau, \gamma$), so that $C_0\tau^{-n}\frac{2\eta}{\gamma}\le\frac{1}{2}$ and $\eta 2^{n/2}\le\frac{\gamma}{4}$, and then $r_0$ so small and $K$ so large, depending on $\eta$, so that $1+Cr_0^{\alpha/2}+Cr_0^{\alpha/2}|\log\eta|\le 2$ and  $C\eta^{1-\frac{n}{2}}(K^{-2}+r_0^{\alpha})^{1/2}\le\frac{\gamma}{4}$, see \eqref{4.65}. By using an upper bound for $r_0$ we get rid of the dependence on it. Therefore $(x,\rho)\in\mathcal{G}(\tau, C_0, 3, r_0)$, completing our proof.
\end{proof}

%\newpage
\section{Local Lipschitz regularity for one-phase almost minimizers}\label{S:onephaselip}

\begin{lemma}\label{L:5.1} Let $u$ be an almost minimizer for $J^+$ in $\Omega$. Let $\theta\in(0,1/2)$. There exist $\gamma>0$, $K_1>1$, $\beta\in (0,1)$ and $r_1>0$ such that if $x\in \Omega$ and $0<r\le r_1$ are such that $B(x,r)\subset \Omega_x$,
\begin{equation}\label{5.1}
b(u_x,x,r)\le \gamma r(1+\omega(u_x,x,r)),
\end{equation}
and
\begin{equation}\label{5.2}
\omega(u_x,x,r)\ge K_1, 
\end{equation}
then
\begin{equation}\label{5.3}
\omega(u_x,x,\theta r)\le \beta \omega(u_x,x,r).
\end{equation}
\end{lemma}

\begin{proof} Recall from the definition of $K_{\text{loc}}^+(\Omega)$ that almost minimizers for $J^+$ are non-negative almost everywhere. Since Theorem \ref{T:continuity} ensures that almost minimizers are continuous (after modification on a set of measure zero), almost minimizers must be non-negative everywhere. Let $x\in\Omega$ and $r\le r_1$ be such that $B(x,r)\subset \Omega_x$. Let $(u_x)_r^*$ denote the harmonic extension of $u_x$ to $B(x,r)$. Notice that, by the maximum principle, $(u_x)_r^*\ge 0$ in $\overline{B}(x,r)$. Given $y\in B(x,r)$, let
\[
a(y)=(u_x)_r^*(x)+\langle \nabla (u_x)_r^*(x),y-x\rangle.
\]
Let also
\begin{equation}\label{5.4}
(v_x)_r^*(y)=(u_x)_r^*(y)-a(y)=(u_x)_r^*(y)-(u_x)_r^*(x)-\langle \nabla (u_x)_r^*(x),y-x\rangle.
\end{equation}
Notice that $(v_x)_r^*$ is harmonic in $B(x,r)$, $(v_x)_r^*(x)=0$ and $\nabla (v_x)_r^*(x)=0$. 
As in \eqref{tt-4}, we obtain that for $0<s\le r$,
\begin{equation}\label{5.5}
\omega(u_x,x,s)\le C\left(\frac{r}{s}\right)^{n/2}r^{\alpha/2}\omega(u_x,x,r)+C\left(\frac{r}{s}\right)^{n/2}+\left(\fint_{B(x,s)}|\nabla (u_x)_r^*|^2\right)^{1/2}.
\end{equation}
We now evaluate $\displaystyle{\fint_{B(x,s)}|\nabla (u_x)_r^*|^2}$. By \eqref{5.4} and because $\nabla a=\nabla (u_x)_r^*(x)$,
\begin{align}
\fint_{B(x,s)}|\nabla (u_x)_r^*|^2&=\fint_{B(x,s)}|\nabla (a+(v_x)_r^*)|^2=\fint_{B(x,s)}|\nabla (v_x)_r^*|^2\nonumber\\\
&+\fint_{B(x,s)}|\nabla a|^2+2\fint_{B(x,s)}\langle \nabla a,\nabla (v_x)_r^*\rangle\nonumber\\\
&=\fint_{B(x,s)}|\nabla (v_x)_r^*|^2+|\nabla (u_x)_r^*(x)|^2+2\langle \nabla (u_x)_r^*(x),\fint_{B(x,s)}\nabla (v_x)_r^*\rangle.\label{5.6}
\end{align}
Since $(v_x)_r^*$ is harmonic in $B(x,r)$, so is $\nabla(v_x)_r^*$, thus $\displaystyle{\fint_{B(x,s)}\nabla (v_x)_r^*=\nabla (v_x)_r^*(x)=0}$. So \eqref{5.6} yields
\begin{equation}\label{5.7}
\fint_{B(x,s)}|\nabla (u_x)_r^*|^2=|\nabla (u_x)_r^*(x)|^2+\fint_{B(x,s)}|\nabla (v_x)_r^*|^2.
\end{equation}
The same proof, with $(x,s)$ replaced with $B(x,r)$, shows that
\begin{equation}\label{5.8}
\fint_{B(x,r)}|\nabla (u_x)_r^*|^2=|\nabla (u_x)_r^*(x)|^2+\fint_{B(x,r)}|\nabla (v_x)_r^*|^2.
\end{equation}

We return to $\displaystyle{\fint_{B(x,s)}|\nabla (u_x)_r^*|^2}$. By \eqref{5.7}, because $\displaystyle{\fint_{B(x,s)}\nabla (v_x)_r^*=\nabla (v_x)_r^*(x)=0}$, by Poincar\'{e}'s inequality and because $\nabla^2 a=0$,
\begin{align}
\fint_{B(x,s)}|\nabla (u_x)_r^*|^2&=|\nabla (u_x)_r^*(x)|^2+\fint_{B(x,s)}|\nabla (v_x)_r^*|^2\nonumber\\
&=|\nabla (u_x)_r^*(x)|^2+\fint_{B(x,s)}\Big|\nabla (v_x)_r^*-\fint_{B(x,s)}\nabla (v_x)_r^*\Big|^2\nonumber\\
&\le |\nabla (u_x)_r^*(x)|^2+Cs^2\fint_{B(x,s)}|\nabla^2 (v_x)_r^*|^2\nonumber\\
&\le |\nabla (u_x)_r^*(x)|^2+Cs^2\fint_{B(x,s)}|\nabla^2 (u_x)_r^*|^2.\label{5.9}
\end{align}
Now suppose that $s<r/2$. By basic properties of harmonic functions,

\begin{align}
\fint_{B(x,s)}|\nabla^2(u_x)_r^*|^2&\le\sup_{B(x,s)}|\nabla^2 (u_x)_r^*|^2\le C\left(r^{-2}\fint_{\partial B(x,r)}|(u_x)_r^*|\right)^2\nonumber\\
&=C\left(r^{-2}\fint_{\partial B(x,r)}u_x\right)^2=Cr^{-4}b(u_x,x,r)^2.\label{5.10}
\end{align}
Now \eqref{5.9} and \eqref{5.10} yield
\begin{equation}\label{5.11}
\fint_{B(x,s)}|\nabla (u_x)_r^*|^2\le  |\nabla (u_x)_r^*(x)|^2+Cr^{-4}s^2b(u_x,x,r)^2.
\end{equation}
By \eqref{5.5} and \eqref{5.11}, since $b(u_x,x,r)\ge 0$ (and because $\sqrt{a^2+b^2}\le a+b$ for $a,b\ge 0$), we have
\begin{equation}
\omega(u_x,x,s)\
\le C\left(\frac{r}{s}\right)^{n/2}r^{\alpha/2}\omega(u_x,x,r)+C\left(\frac{r}{s}\right)^{n/2}+|\nabla (u_x)_r^*(x)|+Cr^{-2}sb(u_x,x,r).\label{5.12}
\end{equation}
Let $\theta\in (0,1/2)$, as in the statement. Take $s=\theta r<r/2$. With this notation, \eqref{5.12} yields, using \eqref{5.2} and \eqref{5.1}:
\begin{align}
\omega(u_x,x,\theta r)&\le |\nabla (u_x)_r^*(x)| +C\theta^{-n/2}r^{\alpha/2}\omega(u_x,x,r)+C\theta^{-n/2}+C\theta r^{-1}b(u_x,x,r)\nonumber\\
&\le |\nabla (u_x)_r^*(x)|+C\theta^{-n/2}(r^{\alpha/2}+K_1^{-1})\omega(u_x,x,r)+C\theta\gamma(1+\omega(u_x,x,r))\nonumber\\
&\le |\nabla (u_x)_r^*(x)|+C\left(\theta^{-n/2}(r^{\alpha/2}+K_1^{-1})+\theta\gamma(K_1^{-1}+1)\right)\omega(u_x,x,r).\label{5.13}
\end{align}
We shall now control $|\nabla (u_x)_r^*(x)|$ in terms of $\omega(u_x,x,r)$. We consider two cases. Let $\eta>0$ be small, to be chosen later. If
\begin{equation}\label{5.14}
\fint_{B(x,r)}|\nabla (v_x)_r^*|^2\ge \eta^2\fint_{B(x,r)}|\nabla u_x|^2=\eta^2\omega(u_x,x,r)^2
\end{equation}
then we use \eqref{5.8} and obtain
\begin{align}
\omega(u_x,x,r)^2&=\fint_{B(x,r)}|\nabla u_x|^2\ge \fint_{B(x,r)}|\nabla (u_x)_r^*|^2=|\nabla (u_x)_r^*(x)|^2+\fint_{B(x,r)}|\nabla (v_x)_r^*|^2\nonumber\\
&\ge |\nabla (u_x)_r^*(x)|^2+\eta^2\omega(u_x,x,r)^2.\label{5.15}
\end{align}
By \eqref{5.13}, \eqref{5.15} yields
\begin{align}
\omega(u_x,x,\theta r)&\le |\nabla (u_x)_r^*(x)|+C\left(\theta^{-n/2}(r^{\alpha/2}+K_1^{-1})+\theta\gamma(K_1^{-1}+1)\right)\omega(u_x,x,r)\nonumber\\
&\le \sqrt{1-\eta^2}\omega(u_x,x,r)+C\left(\theta^{-n/2}(r^{\alpha/2}+K_1^{-1})+\theta\gamma(K_1^{-1}+1)\right)\omega(u_x,x,r).\label{5.16}
\end{align}
Before we continue the analysis of this case, let us deal with the case when \eqref{5.14} fails. In this case, by \eqref{5.8},
\begin{equation}\label{5.17}
\fint_{B(x,r)}|\nabla (u_x)_r^*|^2=|\nabla (u_x)_r^*(x)|^2+\fint_{B(x,r)}|\nabla (v_x)_r^*|^2\le |\nabla (u_x)_r^*(x)|^2+\eta^2\omega(u_x,x,r)^2.
\end{equation}

By standard estimates on harmonic functions,
\begin{equation}\label{5.18}
|\nabla (u_x)_r^*(x)|\le Cr^{-1}\fint_{\partial B(x,r)}|(u_x)_r^*|=Cr^{-1}\fint_{\partial B(x,r)}|u_x|=Cr^{-1}\fint_{\partial B(x,r)}u_x=Cr^{-1}b(u_x,x,r).
\end{equation}
Then, returning to \eqref{5.17},
\begin{align}
\fint_{B(x,r)}|\nabla (u_x)_r^*|^2&\le |\nabla (u_x)_r^*(x)|^2+\eta^2\omega(u_x,x,r)^2\le Cr^{-2}b(u_x,x,r)^2+\eta^2\omega(u_x,x,r)^2\nonumber\\
&\le C\gamma^2(1+\omega(u_x,x,r))^2+\eta^2\omega(u_x,x,r)^2.\label{5.19}
\end{align}
At the same time, \eqref{5.5} with $s=r$, \eqref{5.19} and \eqref{5.2} yield, recalling that $r<r_1$,
\begin{align}
\omega(u_x,x,r)&\le Cr^{\alpha/2}\omega(u_x,x,r)+C+\left(\fint_{B(x,r)}|\nabla (u_x)_r^*|^2\right)^{1/2}\nonumber\\
&\le Cr^{\alpha/2}\omega(u_x,x,r)+C+C\gamma(1+\omega(u_x,x,r))+\eta\omega(u_x,x,r)\nonumber\\
&\le C(r_1^{\alpha/2}+K_1^{-1}+\gamma K_1^{-1}+\gamma+\eta)\omega(u_x,x,r).\label{5.20}
\end{align}
If $\eta$ is small enough so that $C\eta<1/4$ and $K_1$ is large enough and $r_1$ is small enough so that
\begin{equation}\label{5.21}
C(r_1^{\alpha/2}+K_1^{-1}+\gamma K_1^{-1}+\gamma)<\frac{1}{4},
\end{equation}
 we get a contradiction since $\omega(u_x,x,r)\ge K_1>0$. Under these conditions the second case is impossible and \eqref{5.16} holds. 
 
 To deduce \eqref{5.3}, choose $K_1, r_1$ and $\gamma$ satisfying \eqref{5.21} and
 \begin{equation}\label{5.22}
 C\left(\theta^{-n/2}(r_1^{\alpha/2}+K_1^{-1})+\theta\gamma(K_1^{-1}+1)\right)\le\frac{1-\sqrt{1-\eta^2}}{2},
 \end{equation}
 where $\eta$ is as above. Let $\beta\in\left(\frac{1+\sqrt{1-\eta^2}}{2},1\right)$. We have
 \begin{equation}\label{5.23}
 \sqrt{1-\eta^2}+C\left(\theta^{-n/2}(r_1^{\alpha/2}+K_1^{-1})+\theta\gamma(K_1^{-1}+1)\right)\le \beta,
 \end{equation}
 which ensures that \eqref{5.3} holds.
\end{proof}

\begin{theorem}\label{T:5.1}
Let $u$ be an almost minimizer for $J^+$ in $\Omega$. Then $u$ is locally Lipschitz in $\Omega$. 
\end{theorem}

We want to show that there exist $r_2>0$ and $C_2\ge 1$ (depending on $n,\kappa,\alpha,\lambda,\Lambda)$ such that for each choice of $x_0\in \Omega$ and $r_0>0$ such that $r_0\le r_2$ and $B(x_0,K_2r_0)\subset \Omega$, where $K_2$ is as in Lemma \ref{L:4.4}, 
%%%MSVG Here I referred to the K_2 from Lemma \ref{L:4.4}, instead of using the precise formula we had before
 then
\begin{equation}\label{5.24} %%%NN careful doubme label
|u(x)-u(y)|\le C_2(\omega(u_{x_0},x_0,2r_0)+1)|x-y| \ \text{ for } x,y\in B(x_0,r_0).
\end{equation}

\begin{proof} Let $(x,r)$ be such that $B(x,K_2r)\subset \Omega$. We want to use the different Lemmas above to find a pair $(x,\rho)$ that allows us to control $u$. Pick $\theta=1/3$ (smaller values would work as well), and let $\beta, \gamma, K_1, r_1$ be as in Lemma \ref{L:5.1}. 

Pick $\tau=\tau_2/2$, where $\tau_2\in(0,\tau_1)$ where $\tau_1$ is the constant that we get in Lemma \ref{Gselfimprovement} applied with $C_1=3$ and $r_0=r_1$.
Here $\tau_2$ us the corresponding constant that appears in Lemma \ref{L:4.3}.
Let now $r_0,\eta, K$ be as in Lemma \ref{L:4.4} applied to $C_0=10$, and to $\tau$ and $\gamma$ as above. From Lemma \ref{L:4.4} we get a small $r_\gamma$. Set
\begin{equation}\label{5.25} %%%NN careful doubme label
%%%MSVG Here I changed what was K_2 before to K_3, not to confuse with what I called K_2 in Lemma \ref{L:4.4}
K_3\ge \max(K_1,K), \ \ \text{ and } \ \ r_2\le\min(r_1,r_{\gamma}).
\end{equation}
Let $r\le r_2$. We consider three cases.

\textbf{Case 1:}
\begin{equation}\label{5.26}
\begin{cases}
\omega(u_x,x,r)\ge K_3\\
b(u_x,x,r)\ge \gamma r(1+\omega(u_x,x,r))
\end{cases}
\end{equation}

\textbf{Case 2:}
\begin{equation}\label{5.27}
\begin{cases}
\omega(u_x,x,r)\ge K_3\\
b(u_x,x,r)<\gamma r(1+\omega(u_x,x,r))
\end{cases}
\end{equation}

\textbf{Case 3:}
\begin{equation}\label{5.28}
\omega(u_x,x,r)<K_3.
\end{equation}

Let us start with case 1. By \eqref{5.26}, we can apply Lemma \ref{L:4.4} to find $\rho\in\left(\frac{\eta r}{2},\eta r\right)$ such that  $(x,\rho)\in\mathcal{G}(\tau,10,3,r_{\gamma})$. 

Notice that $\tau_1$ obtained in Lemma \ref{Gselfimprovement} depends on an upper bound on $r_0$ (which we had taken to be $r_1$), so if we keep $C_1$ but have a smaller $r_0$, the same $\tau_1$ works. Notice that $\rho<\eta r<r<r_{\gamma}$.
The pair $(x,\rho)$ satisfies the assumptions of Lemmas \ref{Gselfimprovement}-\ref{L:4.3} (applied with $r_0=r_1$), that is, $(x,\rho)\in \mathcal{G}(\tau,10,3,r_{\gamma})$, where $\tau<\tau_2$. By Lemma \ref{L:4.3}, $u$ is $C_x$-Lipschitz in $B(x,\tau r/10)$.

%%MSVG changed $B\left(x,\frac{\tau r\tilde{k}B}{4}\right)$ to $B(x,\tau r/10)$. 

By \eqref{4.37}, we can take
\begin{equation}\label{5.29}
C_x=C(\tau^{-\frac{n}{2}}\omega(u_x,x,\rho)+\rho^{\frac{\alpha}{2}})\le C(\tau^{-\frac{n}{2}}\eta^{-\frac{n}{2}}\omega(u_x,x,r)+r^{\frac{\alpha}{2}}).
\end{equation}
By Lemma \ref{L:4.3} we even know that $u$ is $C^{1,\beta}$ in a neighborhood of $x$, thus Case 1 yields additional regularity.

In the two remaining cases, we set 
\[
r_k=\theta^k r=3^{-k}r, \ k\ge 0.
\]
Our task is to control $\omega(u_x,x,r_k)$. If the pair $(x,r_k)$ ever satisfies \eqref{5.26}, we denote $k_{\text{stop}}$ the smallest integer such that $(x,r_k)$ satisfies \eqref{5.26} (notice that $k\ge 1$ since we are not in Case 1). Otherwise, set $k_{\text{stop}}=\infty$.

Let $k<k_{\text{stop}}$ be given. If $(x,r_k)$ satisfies \eqref{5.27}, we can apply Lemma \ref{5.1} to it. Therefore 
\begin{equation}\label{5.30}
\omega(u_x,x,r_{k+1})\le\beta \omega(u_x,x,r_k).
\end{equation}

Otherwise, $(x,r_k)$ satisfies \eqref{5.28} (since $k<k_{\text{stop}}$). Then
\begin{equation}\label{5.31}
\omega(u_x,x,r_{k+1})=\left(\fint_{B(x,r_{k+1})}|\nabla u_x|^2\right)^{1/2}\le 3^{\frac{n}{2}}\omega(u_x,x,r_k)\le 3^{\frac{n}{2}}K_3.
\end{equation}

%This implies
%\[
%\omega(u,x,\lambda^{1/2}r_{k+1})\le C(n)K_2.
%\]
By \eqref{5.30} and \eqref{5.31}, we obtain that for $0\le k\le k_{\text{stop}}$,
\begin{equation}\label{5.32}
\omega(u_x,x,r_k)\le \max\left(\beta^k\omega(u_x,x,r),3^{\frac{n}{2}}K_3\right).
\end{equation}
%If $k_{\text{stop}}<\infty$, \eqref{5.32} implies 
%\begin{equation}\label{kstop}
%\omega(u_x,x,r_{k_{\text{stop}}})\le 2^n\omega(u_x,x_{k_{\text{stop}-1}})\le 2^n \max\left(\beta^k\omega(u_x,x,r),3^{\frac{n}{2}}K_2\right).
%\end{equation}

If $k_{\text{stop}}=\infty$, this implies that
\begin{equation}\label{5.33}
\limsup_{k\rightarrow\infty} \omega(u_x,x,r_k)\le 3^{\frac{n}{2}}K_3.
\end{equation}
In particular, if $x$ is a Lebesgue point of $\nabla u_x$ (hence a Lebesgue point for $\nabla u$),
\begin{equation}\label{5.34}
|\nabla u_x(x)|\le 3^{n/2}K_3.
\end{equation}
This implies
\begin{equation}\label{5.34'}
|\nabla u(x)|\le C3^{n/2}K_3.
\end{equation}

If $k_{\text{stop}}<\infty$, we apply our argument from Case 1 to the pair $(x,r_{k_{\text{stop}}})$ and get that $u$ is $C^{1,\beta}$ in a neighborhood of $x$. By \eqref{5.29} and \eqref{5.32},
\begin{align}
|\nabla u(x)|&\le C(\tau^{-\frac{n}{2}}\eta^{-\frac{n}{2}}\omega(u_x,x,r_{k_{\text{stop}}})+r_{k_{\text{stop}}}^{\frac{\alpha}{2}})\nonumber\\
&\le C\tau^{-\frac{n}{2}}\eta^{-\frac{n}{2}}\max\left(\beta^{k_{\text{stop}}}\omega(u_x,x,r),3^{\frac{n}{2}}K_3\right)+Cr^{\frac{\alpha}{2}}\nonumber\\
&\le C'\omega(u_x,x,r)+C',\label{5.35}
\end{align}
where $C'$ depends on $n,\kappa, \alpha,\lambda,\Lambda$.
Notice that we still have \eqref{5.35} in Case 1 (directly by \eqref{5.29}), and since \eqref{5.34'} is better than \eqref{5.35}, we proved that if $r\le r_2$, \eqref{5.35} holds for almost every $x\in\Omega$ with $B(x,K_2r)\subset \Omega$.

Now let $x_0\in\Omega$ and $r_0<r_2$ be such that $B(x_0,K_2r_0)\subset \Omega$. Then for almost every $x\in B(x_0,r_0)$, \eqref{5.35} holds with $r=r_0/2$ (so that $B(x,K_2r)\subset B(x_0,K_2r_0)$) and so
\begin{equation}\label{5.36}
|\nabla u(x)|\le C'\omega(u_x,x,r)+C'\le 2^{n/2}C'\omega(u_x,x_0,2r_0)+C'.
\end{equation}
Since we already know that $u$ is in the Sobolev space $W^{1,2}_{\text{loc}}(B(x_0,r_0))$, we deduce from \eqref{5.36} that $u$ is Lipschitz in $B(x_0,r_0)$ and \eqref{5.24} holds, proving Theorem \ref{T:5.1}.
\end{proof}

\section{Almost Monotonicity}\label{S:am}

In this section we establish an analogue of the Alt-Caffarelli-Friedman \cite{ACF} monotonicity formula for variable coefficient almost-minimizers. Recall, for the remainder of this section, the notation $f^{\pm} = \max\{\pm f, 0\}$. In \cite{ACF} it was shown
that the quantity 
\begin{equation}\label{e:defofACF}\begin{aligned} \Phi(f, y, r) &\equiv \frac{1}{r^4}\left(\int_{B(y,r)} \frac{|\nabla f^+|^2}{|z-y|^{n-2}}dz \right)\left(\int_{B(y,r)} \frac{|\nabla f^-|^2}{|z-y|^{n-2}}dz \right)\\
&\equiv \frac{1}{r^4}\Phi_+(f,y,r)\Phi_-(f,y,r)\end{aligned}.\end{equation}
is monotone increasing in $r$ as long as $f(y) = 0$ and $f$ is harmonic. While we cannot expect to get the same monotonicity, we will prove an almost-mononicity result in the style of \cite{davidtoroalmostminimizers}.

\begin{lemma}\label{varphicomputation}
Let $u$ be an almost minimizer for $J$ in $\Omega$, and assume that $B(x,2r) \subset \Omega$, where $x$ is such that $A(x) = I$. Let $\varphi \in W^{1,2}(\Omega)\cap C(\Omega)$ be such that $\varphi(y) \geq 0$ everywhere, $\varphi(y) = 0$ on $\Omega \backslash B(x,r)$, and let $\lambda \in \R$ be such that \begin{equation}\label{lambdacondition} |\lambda \varphi(y)| < 1,\: \mathrm{on}\; \Omega.\end{equation} Then, for each choice of sign, $\pm$, \begin{equation}\label{eqn:varphicomputation} \begin{aligned} 0 \leq Cr^\alpha J_{x,r}(u)+Cr^{\alpha + n} &+ 2\lambda \left[\int_{B(x,r)} \varphi |\nabla u^{\pm}|^2 + \int_{B(x,r)} u^\pm \left\langle \nabla u^{\pm}, \nabla \varphi\right\rangle \right] \\
&+\lambda^2\left[\int_{B(x,r)} \varphi^2|\nabla u^\pm|^2 + (u^\pm)^2|\nabla \varphi|^2 + 2\varphi u^{\pm} \left\langle \nabla u^{\pm}, \nabla \varphi\right\rangle \right], \end{aligned}\end{equation} where $C < \infty$ is a constant which depends only on $\kappa, n, \Lambda, \lambda$ and the $C^{0,\alpha}$ norm of $A$. 
\end{lemma}

\begin{proof}
We verify the proof for $u^+$, the arguments for $u^-$ are similar. Define $v$ on $\Omega$ by \begin{equation}\label{vdefinition} v(y) = u(y) + \lambda \varphi(y)u(y) = (1+\lambda \varphi(y))u^+(y),\; \forall y \in B(x,r)\cap \{u > 0\}\end{equation} and $v(x) \equiv u(x)$ otherwise. It is then easy to verify that $v$ is continuous, that $u$ and $v$ have the same sign on $\Omega$ and that $v^+ = (1+\lambda \varphi)u^+$ everywhere on $\Omega$. We also know that $v^\pm \in W^{1,2}(\Omega)$ with \begin{equation}\label{formulaforvplusgrad} \nabla v^+ = (1+\lambda \varphi)\nabla u^+ + \lambda u^+ \nabla \varphi.\end{equation} For a detailed verification of these facts, see the proof of Lemma 6.1 in \cite{davidtoroalmostminimizers}. 

Because $u$ and $v$ have the same sign and as $\nabla u^- = \nabla v^-$ we can compute that $$J_{x,r}(v) = J_{x,r}(u) + \int_{B(x,r)} \left\langle A\nabla v^+, \nabla v^+\right\rangle - \left\langle A\nabla u^+, \nabla u^+\right\rangle.$$ Also, $u = v$ on $\partial B(x,r)$ so we can use the almost-minimizing properties of $u$ to conclude that \begin{equation}\label{almostminimizingalmostmono} J_{x,r}(u) \leq J_{x,r}(v) + \kappa r^{\alpha + n}.\end{equation} Combining the two above equations we can conclude that \begin{equation}\label{differenceofjuandjv} 0 \leq \kappa r^{\alpha+n}  + \left[\int_{B(x,r)} \left\langle A\nabla v^+, \nabla v^+\right\rangle - \left\langle A\nabla u^+, \nabla u^+\right\rangle \right]\end{equation} Note that $A(x) = I$ and $A$ is H\"older continuous. Thus, on $B(x,r)$, we have $$\left\langle A \nabla v^+, \nabla v^+\right\rangle \leq (1+Cr^\alpha)|\nabla v^+|^2$$ and $$(1-Cr^\alpha)|\nabla u^+|^2 \leq \left\langle A \nabla u^+, \nabla u^+\right\rangle.$$

Using these estimates in \eqref{differenceofjuandjv}, we have \begin{equation}\label{firstlineariezeddifferencebetween}0 \leq  \kappa r^{\alpha + n}+ (1+Cr^\alpha)\left[\int_{B(x,r)} |\nabla v^+|^2 - |\nabla u^+|^2\right] +2Cr^\alpha \int_{B(x,r)} |\nabla u^+|^2.\end{equation}

By the ellipticity of $A$, $r^\alpha \int_{B(x,r)} |\nabla u^+|^2 \leq Cr^{\alpha}J_{x,r}(u)$ and so we get \begin{equation}\label{secondlineariezeddifferencebetween} 0 \leq \kappa r^{\alpha+n} +C_1r^\alpha J_{x,r}(u) + (1+C_2r^\alpha) \left[\int_{B(x,r)} |\nabla v^+|^2 - |\nabla u^+|^2\right], \end{equation} where $C_1, C_2 < \infty$ here depend on $\kappa, \Lambda$ and the $C^{0,\alpha}$ norm of $A$. While the exact values of $C_1, C_2$ are unimportant, we give them subscripts to emphasize that we cannot necessarily take them to be the same constant. 

We note that \eqref{secondlineariezeddifferencebetween} above is very similar to equation (6.14) in \cite{davidtoroalmostminimizers}. We can then argue as in the rest of the proof of Lemma 6.1 there to complete our proof. For the sake of completeness, we include these arguments below. 

By \eqref{formulaforvplusgrad}, \begin{equation}\label{vplusgradexpanded}\begin{aligned} |\nabla v^+|^2 =& (1+\lambda)^2 |\nabla u^+|^2 + 2\lambda(1+\lambda \varphi)u\left\langle \nabla u^+, \nabla \varphi\right\rangle + \lambda^2 (u^+)^2|\nabla \varphi|^2\\ =& |\nabla u^+|^2 + 2\lambda \left[\varphi |\nabla u^+|^2 + u^+\left\langle \nabla u^+, \nabla \varphi\right\rangle \right] \\ &+ \lambda^2 \left[\varphi^2|\nabla u^+|^2 + 2\varphi u^+\left\langle \nabla u^+, \nabla \varphi\right\rangle + (u^+)^2|\nabla \varphi|^2\right]. \end{aligned}\end{equation} Integrate this, place it in \eqref{secondlineariezeddifferencebetween} and get that $$\begin{aligned} 0 \leq &\kappa r^{\alpha + n}+ C_1r^\alpha J_{x,r}(u) + 2\lambda (1+C_2r^\alpha)\left[\int_{B(x,r)}\varphi |\nabla u^+|^2 + u^+\left\langle \nabla u^+, \nabla \varphi\right\rangle \right] \\ &+ \lambda^2(1+C_2r^\alpha) \left[\int_{B(x,r)} \varphi^2|\nabla u^+|^2 + 2\varphi u^+\left\langle \nabla u^+, \nabla \varphi\right\rangle + (u^+)^2|\nabla \varphi|^2\right].\end{aligned}$$

Divide by $(1+C_2r^\alpha)$ and add $(Cr^\alpha- \frac{C_1r^\alpha}{1+C_2r^\alpha})J_{x,r}(u) \geq 0$ for $C$ large enough depending only on $\Lambda, \alpha$ and the $C^{0,\alpha}$ constant of $A$. This gives us the desired inequality \eqref{eqn:varphicomputation}. 
\end{proof}

We will now state and prove variable-coefficient analogues of Lemmas 6.2, 6.3 and 6.4 in \cite{davidtoroalmostminimizers}. We note that the proofs in \cite{davidtoroalmostminimizers} use Lemma 6.1 there, the continuity of almost-minimizers and the logarithmic growth of $\omega(x, r)$. In particular, the proofs go through virtually unchanged for almost-minimizers with variable coefficients. Thus, we will give brief indications of how to adapt the proofs of  \cite{davidtoroalmostminimizers} in our context and invite the reader to study Section 6 in \cite{davidtoroalmostminimizers} for more details.

\begin{lemma}\label{twosidedaplusbound}[Compare to Lemmas 6.2 and 6.3 in \cite{davidtoroalmostminimizers}]
Still assume that $n\geq 3$. Let $u$ be an almost minimizer for $J$ in $\Omega$ and assume that $B(x_0, 4r_0) \subset \Omega$ and that $u(x_0) = 0$ and $A(x_0) = I$. Then, for $0 < r < \min(1, r_0)$ and for each choice of sign, $\pm$, \begin{equation}\label{eqn:twosidedaplusbound}\begin{aligned}
&\left| \frac{c_n}{r^2} \Phi_{\pm}(u, x_0,r) - \frac{1}{n(n-2)}\fint_{B(x_0,r)} |\nabla u^\pm|^2 -\frac{1}{2} \fint_{\partial B(x_0,r)} \left( \frac{u^\pm}{r} \right)^2 \right| \\ &\leq Cr^{\frac{\alpha}{n+1}}\left(1+ \fint_{B(x_0, \tilde{C}r_0)} |\nabla u^2| + \log^2(r_0/r) + \log^2(1/r)\right). \end{aligned}
\end{equation}
Again, $c_n = (n(n-2)\omega_n)^{-1}$ and $C > 0$ depending only on $n, \Lambda, \lambda, \|A\|_{C^{0,\alpha}}$ and the almost-minimizing constants of $u$. 
\end{lemma}

\begin{proof}
We will prove this for $u^+$ and only prove the lower bound on the left hand side of \eqref{eqn:twosidedaplusbound}. The modifications required to prove the upper bound and the statement for $u^-$ are exactly as in \cite{davidtoroalmostminimizers} (see, in particular, Lemma 6.3 there) and we leave them to the interested reader. 

Fix $s < r$ and apply Lemma \ref{varphicomputation} with \begin{equation}\label{onesideddefofphi} \varphi(y) \equiv \varphi_{r,s}(y) \equiv \left\{ \begin{array}{cc} 0 & \:\mathrm{for}\: y \in \Omega \backslash B(x_0,r)\\
c_n\left(|y-x_0|^{2-n} - r^{2-n}\right) & \: \mathrm{for}\: y \in B(x_0,r) \backslash B(x_0,s)\\
c_n s^{2-n} - c_n r^{2-n} & \: \mathrm{for}\: y\in B(x_0,s),\\
\end{array} \right.
\end{equation}
where the constant $c_n$ is such that $\int_{\partial B(x_0, r)} \partial_{\hat{n}} \varphi_{r,s}(y) = 1$. Finally, let $\lambda = c_n^{-1}r^{n-2+\frac{n\alpha}{n+1}}$ and $s = r^{1+\frac{\alpha}{2(n+1)}}$. 

Inserting this choice for $\varphi, \lambda$ into \eqref{eqn:varphicomputation}, integrating by parts (moving the derivative onto the $\varphi$ term) and using Cauchy-Schwartz we get

\begin{equation}\label{eqn:varphirscomputation} \begin{aligned} 0 \leq Cr^\alpha J_{x_0,r}(u)+Cr^{\alpha + n} &+ 2\lambda \left[\int_{B(x_0,r)} \varphi |\nabla u^{\pm}|^2 - \frac{1}{2}\left(\fint_{\partial B(x_0,r)} (u^\pm)^2 - \fint_{\partial B(x_0,s)} (u^{\pm})^2\right)\ \right] \\
&+2\lambda^2\left[\int_{B(x_0,r)} \varphi^2|\nabla u^\pm|^2 +(u^{\pm})^2|\nabla \varphi|^2 \right]. \end{aligned}\end{equation}

Using the definition of $\varphi$ and the estimates $$\|\varphi\|_\infty \leq \frac{c_n}{s^{n-2}} \qquad \text{and} \qquad \|\nabla \varphi\|_\infty \leq \frac{c_n(n-2)}{s^{n-1}},$$ we can deduce that \begin{equation}\label{eqn:varphirscomputation2} \begin{aligned} 0 \leq Cr^\alpha J_{x_0,r}(u)+Cr^{\alpha + n} &+ 2\lambda \left[\Phi_+(u, x_0, r) - \frac{c_n}{r^{n-2}} \int_{B(x,r)}|\nabla u^+|^2 \right] \\
&-\lambda \left(\fint_{\partial B(x_0,r)} (u^\pm)^2 - \fint_{\partial B(x_0,s)} (u^{\pm})^2\right) \\
&+2\lambda^2\left[\frac{c_n^2}{s^{2n-4}}\int_{B(x_0,r)} |\nabla u^\pm|^2 +\frac{c_n^2(n-2)^2}{s^{2n-2}} \int_{B(x_0, r)\setminus B(x_0, s)}(u^{\pm})^2 \right]. \end{aligned}\end{equation}

We want to estimate $$\begin{aligned} M \equiv& \frac{c_n}{r^2} \Phi_+(u, x_0, r) - \frac{1}{n(n-2)}\fint_{B(x_0,r)} |\nabla u^+|^2 - \frac{1}{2}\fint_{\partial B(x_0, r)} \left(\frac{u^+}{r}\right)^2\\
\equiv&\frac{c_n}{r^2} \Phi_+(u, x_0, r) - c_nr^{-n} \int_{B(x_0, r)} |\nabla u^+|^2 - \frac{1}{2r^2}\fint_{\partial B(x_0, r)} (u^+)^2.\end{aligned}$$

Rearranging the terms of \eqref{eqn:varphirscomputation2} and dividing by $2\lambda r^2$ we get that \begin{equation}\label{e:upperM}\begin{aligned} -M \leq& \frac{Cr^\alpha\left(J_{x_0, r}(u) + r^n\right)}{\lambda r^2} + \frac{1}{2r^2} \fint_{B(x_0, s)} (u^+)^2\\ +& \lambda r^{-2} \left[\frac{c_n^2}{s^{2n-4}}\int_{B(x_0,r)} |\nabla u^\pm|^2 +\frac{c_n^2(n-2)^2}{s^{2n-2}} \int_{B(x_0, r)\setminus B(x_0, s)}(u^{\pm})^2 \right]. \end{aligned}\end{equation}
%I = -\varphi(s)\fint_{\partial B(x_0,s)} (u^{\pm})^2  - u^2 |\nabla \varphi|^2
By the continuity of $u$, more specifically the last estimate in the proof of Theorem \ref{T:continuity}, and $u(x_0) = 0$ we can estimate \begin{equation}\label{e:boundaryest}\begin{aligned}
\fint_{B(x_0, s)} (u^+)^2 \leq&Cs^2\left(\omega(u, x_0, 2r_0) + +\log\left(\frac{r_0}{s}\right)\right)^2,\\
\int_{B(x_0, r)\setminus B(x_0, s)}(u^{\pm})^2 \leq&Cr^{n+2}\left(\omega(u, x_0, 2r_0) +\log\left(\frac{r_0}{s}\right)\right)^2.\end{aligned}
\end{equation}

Apply the estimates \eqref{e:boundaryest} to the corresponding terms in \eqref{e:upperM} and use the logarithmic growth of the Dirichlet energy, \eqref{need}, to bound both the energy term in $J_{x_0, r}$ and the term $$\int_{B(x_0, r)} |\nabla u^+|^2 \leq \int_{B(x_0, r)} |\nabla u|^2 \leq Cr^n\left(\omega(u, x_0, r_0) + \log(r_0/r)\right)^2.$$ Note that to bound the energy term in $J_{x_0, r}$ we need to use the ellipticity of $A$. Finally, overestimate the area terms in $J$ by $|B(x, r)| \simeq r^n$. After some arithmetic, and plugging in the values for $\lambda, s$ we arrive at the desired result.  See \cite{davidtoroalmostminimizers} for the detailed computations. 

\end{proof}

The next two results follow from the previous theorems just as they do in \cite{davidtoroalmostminimizers}. We state them here without proof and encourage the reader to refer to \cite{davidtoroalmostminimizers} for full details.

\begin{lemma}\label{averageenergybound}[Compare to Lemma 6.4 in \cite{davidtoroalmostminimizers}]
Let $u$ be an almost minimizer for $J$ in $\Omega$, and assume that $B(x_0, 4r_0) \subset \Omega$ with $u(x_0) = 0$ and $A(x_0) = I$. For $0 < r < \frac{1}{2}\min(1, r_0)$, set $t \equiv t(r) \equiv \left(1 - \frac{r^{\alpha/4}}{10}\right)r$. Then for $0 < r < \min(1/2, r_0)$ and each choice of sign, $\pm$, \begin{equation}\begin{aligned} &\left|\fint_{t(r)}^r \left(\int_{B(x_0,s)}|\nabla u^\pm(y)|^2 dy\right)ds - \fint_{t(r)}^r \left(\int_{\partial B(x_0,s)} u^\pm \frac{\partial u^\pm}{\partial n}\right)ds\right| \\ &\leq Cr^{n+\alpha/4}\left(1 + \fint_{B(x_0, \tilde{C}r_0)}|\nabla u|^2 + \log^2\frac{r_0}{r}\right). \end{aligned}\end{equation}

Here, $\partial u^\pm/\partial n$ denotes the radial derivative of $u^\pm$ and $C >0$ depend only on $\|q_{\pm}\|_{\infty}, n, \Lambda$, $\lambda, \|A\|_{C^{0,\alpha}}$ and the almost-minimization constants. \end{lemma}

\begin{theorem}\label{almostmonotonicity}
Let $u$ be an almost minimizer for $J$ in $\Omega$ and let $\delta$ be such that $0 < \delta < \alpha/4(n+1)$. Let $B(x_0, 4r_0) \subset \Omega$ with $u(x_0) = 0$ and $A(x_0) = I$.  Then there exists $C > 0$, depending on the usual parameters such that for $0 < s < r < \frac{1}{2}\min(1, r_0)$,
\begin{equation}\label{eqnalmostmono}
\Phi(u, x_0, s) \leq \Phi(u, x_0, r) + C(x_0, r_0)r^\delta,
\end{equation}
where,
\begin{equation}\label{defofconstant}
C(x_0,r_0) \equiv C + C\left(\fint_{B(x_0, 2r_0)}|\nabla u|^2\right)^2 + C((\log r_0)_+)^4.
\end{equation}
\end{theorem}

\section{Local Lipschitz continuity for two-phase almost minimizers}\label{sec:twophaselipschitz}

The proof of two-phase Lipschitz continuity follows the same blue-print as the one-phase case. We start with Lemma \ref{twosidedlipschitzlemma} which is an analogue of Lemma \ref{L:5.1}. However, the proof of  Lemma \ref{twosidedlipschitzlemma} is a bit more involved as it requires the use of the two-phase almost monotonicity formula, \eqref{eqnalmostmono}.

\begin{lemma}\label{twosidedlipschitzlemma}
Let $u$ be an almost minimizer for $J$ in $\Omega$ and let $B_0 \equiv B(x_0, \lambda^{-1/2}r_0) \subset \Omega$ be given. Let $\theta \in (0,1/3)$ and $\beta \in (0,1)$. Then there exists $\gamma > 0, K_1 > 1$ and $r_1 > 0$ (which may depend on $\theta$ and $\beta$) such that if $x \in B(x_0, r_0)$ and $0 < r \leq r_1$ satisfy \begin{equation}\label{uiszerosomewhere} u_x(y) = 0 \mbox{ for some } y \in B(x, 2r/3), \end{equation}
%\vspace{-1cm}
\begin{equation}\label{bissmall}
|b(u_x,x,r)| \leq \gamma r(1+\omega(u_x, x,r)), \mbox{ and}
\end{equation}
%\vspace{-1cm}
\begin{equation}\label{omegaisbig}
\omega(u_x, x,r) \geq K_1.
\end{equation}
Then, \begin{equation}\label{omegadecays}
\omega(u_x,x, \theta r) \leq \beta \omega(u_x, x,r).
\end{equation}
\end{lemma}

\begin{proof}
Let $x,r$ be as in the statement, and $y \in B(x, 2r/3)$ such that $u_x(y) = 0$. As usual, let $(u_x)_r^*$ be the harmonic extension of $u_x$ to $ B(x,r)$. By standard elliptic estimates there exists a $c > 0$ (depending on $\theta \in (0,1/3)$ but independent of $r, x$) such that for all $z\in B(x, \theta r)$, $$|\nabla (u_x)_r^*(z)|\leq \frac{c}{r}\sup_{\zeta\in \partial B(x, 2\theta r)} |(u_x)_r^*(\zeta)| \leq \frac{cb^+(u_x, x, r)}{r}.$$  Using this estimate, \eqref{2.4mult} and the triangle inequality we can say:

\begin{equation}\label{e:triangletwophase}
\begin{aligned}
\omega^2(u_x, x, \theta r) \leq& 2\omega^2((u_x)_r^*, x, \theta r) + \fint_{B(x, \theta r)} |\nabla \left((u_x)_r^* - u_x\right)|^2\\
\leq& C\left(\frac{b^+(u_x, x, r)}{r}\right)^2 + C\fint_{B(x, r)} |\nabla \left((u_x)_r^* - u_x\right)|^2\\
\leq& C\left(\frac{b^+(u_x, x, r)}{r}\right)^2 + C + Cr^\alpha \omega^2(u_x, x, r),
\end{aligned}
\end{equation}
where $C > 0$ depends on the dimension and the almost-minimization properties of $u_x$, but, crucially, not on $K_1$. 

For any $\beta \in (0,1)$ if $K_1$ is large enough and $r_1$ is small enough (depending on $\beta$), then (using condition \eqref{omegaisbig}),  $$C + Cr^\alpha \omega^2(u_x, x, r) \leq \frac{\beta^2 \omega^2(u_x, x, r)}{2}.$$ Thus to prove \eqref{omegadecays}, it suffices to bound \begin{equation}\label{e:reductiontobplus}C\left(\frac{b^+(u_x, x, r)}{r}\right)^2 \leq \frac{\beta^2 \omega^2(u_x, x, r)}{2}.\end{equation}

Recall the notation $u^\pm_x := \max\{\pm u_x, 0\}$. To simplify our exposition, we need to specify whether $u^+_x$ or $u^-_x$ contributes more to the energy around $x$ at scale $r$ (of course the two situations are symmetric). So assume, without loss of generality, that $$\omega(u^+_x, x,r) \equiv \left(\fint_{B(x,r)} |\nabla u^+_x|^2\right)^{1/2} \leq  \left(\fint_{B(x,r)} |\nabla u^-_x|^2\right)^{1/2} \equiv \omega(u^-_x, x,r).$$

We will now bound $b^+(u_x, x,r)$ by $\omega(u_x^+, x, r)$. We then finish by bounding $\omega(u_x^+, x, r)$ by a constant depending on $x_0, r_0$. This requires the monotonicity formula we developed in the previous section. 

To begin, note that $$\frac{b^+(u_x, x,r)}{r} \leq \frac{2}{r}\fint_{\partial B(x,r)} u^+_x - \frac{1}{r} b(u_x, x,r) \stackrel{\eqref{bissmall}}{\leq} \frac{2}{r}\fint_{\partial B(x,r)} u^+_x + \gamma(1+\omega(u_x, x,r)).$$ Choosing $\gamma > 0$ small (depending on $\beta$ and $K_1$), the second term on the right hand side above is dominated by $\frac{\beta \omega(u_x, x, r)}{8}$. So we have further simplified the problem and now it suffices to bound $$\frac{2}{r}\fint_{\partial B(x,r)} u^+_x \leq \frac{\beta \omega(u_x, x, r)}{2}.$$

Recall that $y\in B(x,2r/3)$ is such that $u_x(y) = 0$. Fix $\eta > 0$ small but to be determined later, and let $z\in B(y, \eta r/8)$. In particular $B(z,\eta r)\subset B(x,r)$. For such $z$ integrate on rays from points in $\partial B(z, \eta r)$ to points in $\partial B(x,r)$ and using Fubini we see that \begin{equation}\label{e:boundbypointsclosetoz}\begin{aligned} \fint_{\partial B(x,r)} u^+_x \leq& \fint_{\partial B(z,\eta r)} u^+_x + C(\eta) r\fint_{B(x,r)} |\nabla u_x^+|\\
\leq& \underbrace{\sup_{\partial B(z, \eta r)} u^+_x}_{I} + C(\eta) r\underbrace{\left(\fint_{B(x,r)}|\nabla u_x^+|^2\right)^{1/2}}_{II}.
\end{aligned}
\end{equation}

Term $I$ in \eqref{e:boundbypointsclosetoz} is small because points in $\partial B(z, \eta r)$ are close to $y$, $u_x(y)=0$ and $u_x$ does not oscillate too much. To wit, by Theorem \ref{T:continuity} applied inside the ball $B(x,r)$ (more specifically, using the penultimate equation in the Theorem's proof), for all $\zeta \in \partial B(z, \eta r)$, 
$$\begin{aligned} |u^+_x(\zeta)| =& |u^+_x(\zeta)-u_x^+(y)|\le C|\zeta-y|\left( 1+\omega(u_x,x, r) +\log\left(\frac{r}{|\zeta-y|}\right) \right)\\ \leq& C\eta r\left(1 + \omega(u_x, x,r) + \log\left(\frac{1}{\eta}\right)\right).\end{aligned}$$ Picking $\eta > 0$ small enough (again depending only on $K_1$ and $\beta$) this allows us to bound $I$ by $r \beta \omega(u_x,x, r)/8$ as desired.

To bound $II$ in \eqref{e:boundbypointsclosetoz} note that 
\begin{equation}\label{e:usemono}
\begin{aligned} \omega(u_x^+, x,r)^2\omega(u_x^-, x, r)^2 
\stackrel{\eqref{2.14g}}{\leq}& %%%  essayer ca, c'\'etait wcomparison
C\omega^2(u^+, x ,\Lambda^{1/2}r)\omega^2(u^-, x, \Lambda^{1/2}r)\\ \leq& C\omega^2(u^+, y, (1+\Lambda^{1/2})r)\omega^2(u^-, y, (1+\Lambda^{1/2})r)\\ 
%\stackrel{\eqref{wcomparison2}}{\leq} %%%GG faudra changer la
\leq & 
C\omega^2(u^+_y, y, \lambda^{-1/2}(1+\Lambda^{1/2})r)\omega^2(u^-_y, y, \lambda^{-1/2}(1+\Lambda^{1/2})r)\\ \leq& C\Phi(u_y, y,\lambda^{-1/2}(1+\Lambda^{1/2})r)\\  
\stackrel{\eqref{eqnalmostmono}}{\leq}& 
C\Phi(u_y, y, (100+\Lambda)^{-1/2}r_0) + C r_0^\delta,
\end{aligned}
\end{equation} 
%c'etait ca comparaison2
%\begin{align}
%\omega&(u,x,s)=\left( \fint_{B(x,s)}|\nabla u|^2\right)^{1/2}\le \left(\lambda^{-1}\fint_{B(x,s)}\langle A(x)\nabla u,\nabla u\rangle\right)^{1/2}\nonumber\\
%&=\left(\lambda^{-1}\fint_{T_x(B(x,s))}|\nabla u_x|^2\right)^{1/2}\le  C\left(\fint_{B(x,\lambda^{-1/2}s)}|\nabla u_x|^2\right)^{1/2}=C\omega(u_x,x,\lambda^{-1/2}s).\label{wcomparison2}
%\end{align}
where $C >0$ depends on the $\fint_{B(x_0,2r_0)} |\nabla u|^2, r_0$ and the almost-minimization constants of $u$ but crucially does not depend on $x, r$ or $y$. In what remains, we will denote by $C(B_0)$ constants that are uniform over points and scales inside of $B(x_0, 2r_0)$.

Recall, from above that $$\begin{aligned} \Phi(f, y, r) &\equiv \frac{1}{r^4}\left(\int_{B(y,r)} \frac{|\nabla f^+|^2}{|z-y|^{n-2}}dz \right)\left(\int_{B(y,r)} \frac{|\nabla f^-|^2}{|z-y|^{n-2}}dz \right)\\
&\equiv \frac{1}{r^4}\Phi_+(f,y,r)\Phi_-(f,y,r)\end{aligned}.$$ 

For ease of notation let $c_1 =(100+\Lambda)^{-1/2}$, so that $B(y, c_1 r_0), B(y, \Lambda^{1/2}c_1r_0) \subset B(y, r_0) \subset B(x_0,2r_0)$. We estimate 
\begin{equation}\label{apmbound}
\begin{aligned}\frac{\Phi_{\pm}(u_y, y, c_1r_0)}{(c_1r_0)^2} 
=& \sum_{i=0}^\infty 2^{-2i}\frac{1}{(2^{-i}c_1r_0)^2}(\Phi_\pm(u_y, y, 2^{-i}c_1r_0) - \Phi_{\pm}(u_y, y, 2^{-i-1}c_1r_0))\\
 \leq& C \sum_{i=0}^\infty 2^{-2i} \fint_{B(y, 2^{-i}c_1r_0)} |\nabla u_y|^2 
 \stackrel{\eqref{2.11mult}}{\leq} 
 C\sum_{i=0}^\infty 2^{-2i}\left(\omega(u_y,y, c_1r_0)+i\right)^2\\
\leq& C(\omega(u_y, y, c_1r_0) + 1)^2 
\stackrel{\eqref{2.14g}}{\leq}  %%%  essayer ca, c'\'etait wcomparison
C(\omega(u, y,r_0) + 1)^2\\ \leq& C(\omega(u, x_0, 2r_0) + 1)^2.
\end{aligned}\end{equation}

Combine \eqref{e:usemono} and \eqref{apmbound} to obtain, $$\omega(u_x^+, x,r)^2\omega(u_x^-, x, r)^2 \leq CC(B_0) \Rightarrow \omega(u_x^+, x, r) \leq \sqrt{CC(B_0)}.$$

Continuing we see that $$\omega(u_x^+, x,r)^2+\omega(u_x^-, x, r)^2 \geq K_1^2 \Rightarrow \omega(u_x^-, x, r)^2 \geq K_1^2 - CC(B_0) \geq K_1^2/2,$$ if $K_1 > 0$ is large enough (but chosen uniformly over $B_0$). Putting the above two offset equations together we have \begin{equation}\label{e:finalboundonwx}\begin{aligned} \frac{K_1^2\omega(u_x^+, x,r)^2}{2} \leq& CC(B_0)\\
\Rightarrow \omega(u_x^+, x, r) \leq& \frac{\sqrt{2CC(B_0)}}{K_1} \leq \frac{\beta \omega(u_x, x, r)}{16},\end{aligned}\end{equation} where the last inequality is again justified by choosing $K_1$ large enough depending on $\beta \in (0,1)$ and uniform over $B_0$. This completes the bound of $II$ in \eqref{e:boundbypointsclosetoz} and in turn completes the proof. 
\end{proof}

\begin{theorem}\label{T:7.1}
Let $u$ be an almost minimizer for $J$ in $\Omega$. Then $u$ is locally Lipschitz in $\Omega$. 
\end{theorem}

Again, we will show a more precise estimate;  that there exist $r_2>0$ and $C_2\ge 1$ (depending on $n,\kappa,\alpha,\lambda,\Lambda)$ such that for each choice of $x_0\in \Omega$ and $r_0>0$ such that $r_0\le r_2$ and $B(x_0,K_2r_0)\subset \Omega$ (with $K_2$ as in Lemma \ref{L:4.4}), then
\begin{equation}\label{7.25}
|u(x)-u(y)|\le C_2(\omega(u_{x_0},x_0,2r_0)+1)|x-y| \ \text{ for } x,y\in B(x_0,r_0).
\end{equation}

\begin{proof} Let $(x,r)$ be such that $B(x,K_2r)\subset \Omega$. We want to use the different Lemmas above to find a pair $(x,\rho)$ that allows us to control $u$. Pick $\theta=1/3, \beta = 1/2$ (smaller values would work as well), and let $\gamma, K_1, r_1$ be as in Lemma \ref{twosidedlipschitzlemma}. 

Pick $\tau=\tau_2/2$, where $\tau_2\in(0,\tau_1)$ where $\tau_1$ is the constant that we get in Lemma \ref{Gselfimprovement} applied with $C_1=3$ and $r_0=r_1$.
Here $\tau_2$ us the corresponding constant that appears in Lemma \ref{L:4.3}.
Let now $r_0,\eta, K$ be as in Lemma \ref{L:4.4} applied to $C_0=10$, and to $\tau$ and $\gamma$ as above. From Lemma \ref{L:4.4} we get a small $r_\gamma$. Set
%%%MSVG I changed what was called K_2 to K_3 here, not to confuse with what we are now calling K_2 from Lemma \ref{L:4.4}
\begin{equation}\label{e:kandrtwophase}
K_3\ge \max(K_1,K), \ \ \text{ and } \ \ r_2\le\min(r_1,r_{\gamma}).
\end{equation}
Let $r\le r_2$. We consider four cases.

\textbf{Case 0:}
\begin{equation}\label{e:novanish}
u_x(z) \neq 0,\qquad  \forall z \in B(x, 2r/3)
\end{equation}

\textbf{Case 1:} $u_x(z) = 0$ for some $z \in  B(x, 2r/3)$ and
\begin{equation}\label{7.26}
\begin{cases}
\omega(u_x,x,r)\ge K_3\\
b(u_x,x,r)\ge \gamma r(1+\omega(u_x,x,r))
\end{cases}
\end{equation}

\textbf{Case 2:} $u_x(z) = 0$ for some $z \in  B(x, 2r/3)$ and
\begin{equation}\label{7.27}
\begin{cases}
\omega(u_x,x,r)\ge K_3\\
b(u_x,x,r)<\gamma r(1+\omega(u_x,x,r))
\end{cases}
\end{equation}

\textbf{Case 3:} $u_x(z) = 0$ for some $z \in  B(x, 2r/3)$ and
\begin{equation}\label{7.28}
\omega(u_x,x,r)<K_3.
\end{equation}

Let us start with {\bf Case 0}. If $u_x$ does not vanish inside of $B(x, 2r/3)$ then we know that $u$ is $C^{1,\beta}$ in a neighborhood of $x$ and \eqref{limnabla} %%% essayer ca. C'\'etait equiv3.18
tells us
\begin{equation}\label{e:novanishreg}|\nabla u_x(y)|\le C\left(\omega(u_x,x,r)+r^{\alpha/2}\right) \ \text{ for almost every } y\in B(x,10^{-3}\lambda^{1/2}\Lambda^{-1/2}r).\end{equation}

If we are in case {\bf Case 1}, by \eqref{7.26}, we can apply Lemma \ref{L:4.4} to find $\rho\in\left(\frac{\eta r}{2},\eta r\right)$ such that  $(x,\rho)\in\mathcal{G}(\tau,10,3,r_{\gamma})$. 

Notice that $\tau_1$ obtained in Lemma \ref{Gselfimprovement} depends on an upper bound on $r_0$ (which we had taken to be $r_1$), so if we keep $C_1$ but have a smaller $r_0$, the same $\tau_1$ works. Notice that $\rho<\eta r<r<r_{\gamma}$.
The pair $(x,\rho)$ satisfies the assumptions of Lemmas \ref{Gselfimprovement}-\ref{L:4.3} (applied with $r_0=r_1$), that is, $(x,\rho)\in \mathcal{G}(\tau,10,3,r_{\gamma})$, where $\tau<\tau_1$. By Lemma \ref{L:4.3}, $u$ is $C_x$-Lipschitz in $B(x,\tau r/10)$.

%%%MSVG changed $B\left(x,\frac{\tau r\tilde{k}B}{4}\right)$ into $B(x,\tau r/10)$ above. 

By \eqref{4.37}, we can take
\begin{equation}\label{7.29}
C_x=C(\tau^{-\frac{n}{2}}\omega(u_x,x,\rho)+\rho^{\frac{\alpha}{2}})\le C(\tau^{-\frac{n}{2}}\eta^{-\frac{n}{2}}\omega(u_x,x,r)+r^{\frac{\alpha}{2}}).
\end{equation}
By Lemma \ref{L:4.3} we even know that $u$ is $C^{1,\beta}$ in a neighborhood of $x$, thus Case 1 yields additional regularity.

In the two remaining cases, we set 
\[
r_k=\theta^k r=3^{-k}r, \ k\ge 0.
\]
Our task is to control $\omega(u_x,x,r_k)$. If the pair $(x,r_k)$ ever satisfies \eqref{7.26} or \eqref{e:novanish} we denote $k_{\text{stop}}$ the smallest integer such that $(x,r_k)$ satisfies \eqref{7.26} or  \eqref{e:novanish} (notice that $k\ge 1$ since we are not in Cases 0 or 1). Otherwise, set $k_{\text{stop}}=\infty$.

Let $k<k_{\text{stop}}$ be given. If $(x,r_k)$ satisfies \eqref{7.27} and there exists a $y\in B(x, 2r_k/3)$ such that $u_x(y) = 0$, we can apply Lemma \ref{twosidedlipschitzlemma} at that point and scale. Therefore 
\begin{equation}\label{7.30}
\omega(u_x,x,r_{k+1})\le \omega(u_x,x,r_k)/2.
\end{equation}

Otherwise, $(x,r_k)$ satisfies \eqref{7.28} (since $k<k_{\text{stop}}$). Then
\begin{equation}\label{7.31}
\omega(u_x,x,r_{k+1})=\left(\fint_{B(x,r_{k+1})}|\nabla u_x|^2\right)^{1/2}\le 3^{\frac{n}{2}}\omega(u_x,x,r_k)\le 3^{\frac{n}{2}}K_3.
\end{equation}

%This implies
%\[
%\omega(u,x,\lambda^{1/2}r_{k+1})\le C(n)K_2.
%\]
By \eqref{7.30} and \eqref{7.31}, we obtain that for $0\le k\le k_{\text{stop}}$,
\begin{equation}\label{7.32}
\omega(u_x,x,r_k)\le \max\left(2^{-k}\omega(u_x,x,r),3^{\frac{n}{2}}K_3\right).
\end{equation}
%If $k_{\text{stop}}<\infty$, \eqref{5.32} implies 
%\begin{equation}\label{kstop}
%\omega(u_x,x,r_{k_{\text{stop}}})\le 2^n\omega(u_x,x_{k_{\text{stop}-1}})\le 2^n \max\left(\beta^k\omega(u_x,x,r),3^{\frac{n}{2}}K_2\right).
%\end{equation}

If $k_{\text{stop}}=\infty$, this implies that
\begin{equation}\label{7.33}
\limsup_{k\rightarrow\infty} \omega(u_x,x,r_k)\le 3^{\frac{n}{2}}K_3.
\end{equation}
In particular, if $x$ is a Lebesgue point of $\nabla u_x$ (hence a Lebesgue point for $\nabla u$),
\begin{equation}\label{7.34}
|\nabla u_x(x)|\le 3^{n/2}K_3.
\end{equation}
This implies
\begin{equation}\label{7.34'}
|\nabla u(x)|\le C3^{n/2}K_3.
\end{equation}

If $k_{\text{stop}}<\infty$, we apply our argument from either Case 0 or Case 1 to the pair $(x,r_{k_{\text{stop}}})$ and get that $u$ is $C^{1,\beta}$ in a neighborhood of $x$. By either \eqref{e:novanishreg} or \eqref{7.29} and then \eqref{7.32},
\begin{align}
|\nabla u(x)|&\le C(\omega(u_x,x,r_{k_{\text{stop}}})+r_{k_{\text{stop}}}^{\frac{\alpha}{2}})\nonumber\\
&\le C\max\left(\beta^{k_{\text{stop}}}\omega(u_x,x,r),3^{\frac{n}{2}}K_3\right)+Cr^{\frac{\alpha}{2}}\nonumber\\
&\le C'\omega(u_x,x,r)+C',\label{7.35}
\end{align}
where $C'$ is independent of $x, r$. %\Max{be more explicit here?}%$n,\kappa, \alpha,\lambda,\Lambda$.
Notice that we still have \eqref{7.35} in Cases 0 or 1 (directly by \eqref{e:novanishreg} or \eqref{7.29}), and since \eqref{7.34'} is better than \eqref{7.35}, we proved that if $r\le r_2$, \eqref{7.35} holds for almost every $x\in\Omega$ with $B(x,K_2r)\subset \Omega$.

Now let $x_0\in\Omega$ and $r_0<r_2$ be such that $B(x_0,K_2r_0)\subset \Omega$. Then for almost every $x\in B(x_0,r_0)$, \eqref{7.35} holds with $r=r_0/2$ (so that $B(x,K_2r)\subset B(x_0,K_2r_0)$ and so
\begin{equation}\label{7.36}
|\nabla u(x)|\le C'\omega(u_x,x,r)+C'\le 2^{n/2}C'\omega(u_x,x_0,2r_0)+C'.
\end{equation}
Since we already know that $u$ is in the Sobolev space $W^{1,2}_{\text{loc}}(B(x_0,r_0))$, we deduce from \eqref{7.36} that $u$ is Lipschitz in $B(x_0,r_0)$ and \eqref{7.25} holds, proving Theorem \ref{T:7.1}.
\end{proof}

%\begin{tabular}{l}
%Guy David\\
%Equipe d'Analyse Harmonique\\
%Universit\'e Paris-Sud\\
%Batiment 425 \\
%91405 Orsay Cedex, France
%\\ {\small \tt Guy.David@math.u-psud.fr}
%\hfill
%\end{tabular}
%\begin{tabular}{l}
%Max Engelstein\\
%MIT\\
%Department of Mathematics\\
%77 Massachusetts Avenue\\
%Cambridge, MA 02139-4307 USA\\
%{\small \tt maxe@mit.edu}\\
%\end{tabular}

%\vspace{4mm}
%\begin{tabular}{l}
%Mariana Smit Vega Garcia \\ Western Washington University \\ Department of Mathematics \\
%BH 230\\
%Bellingham, WA 98225, USA
%\\ {\small \tt Mariana.SmitVegaGarcia@wwu.edu}\\
%\end{tabular}
%\begin{tabular}{l}
%Tatiana Toro \\ University of Washington \\ Department of Mathematics \\ Box 354350 \\
%Seattle, WA 98195-4350 USA
%\\ {\small \tt toro@math.washington.edu}\\
%\end{tabular}

\end{document}